\documentclass[11pt]{article}

\usepackage[margin=2.3cm]{geometry}
\usepackage{amsmath,amssymb,amsthm}
\usepackage{enumitem,color}
\usepackage{hyperref}

\usepackage{graphicx,subfigure}
\usepackage{multirow,booktabs,verbatim}
\newcommand{\p}{\partial}
\newcommand{\ds}{\displaystyle}
\newcommand{\f}{\frac}
\usepackage{bm}
\newtheorem{theorem}{Theorem}[section]
\newtheorem{lemma}[theorem]{Lemma}
\newtheorem{corollary}[theorem]{Corollary}
\newtheorem{remark}[theorem]{Remark}

\title{Error estimate of the nonuniform BDF3--L2 method for subdiffusion equations via multiscale solution decomposition}

\author{
Wenlin Qiu\thanks{School of Mathematics, Yunnan Normal University, Kunming 650500, China. Email: qwllkx12379@163.com.}
\and Kexin Li\thanks{Corresponding author. School of Statistics and Mathematics, Yunnan University of Finance and Economics, Kunming 650221, China. Email: likx1213@163.com.}
\and Yiqun Li\thanks{Corresponding author. School of Mathematics and Statistics, Wuhan University, Wuhan 430072, China. Email: YiqunLi24@outlook.com.}
\and Hao Zhang\thanks{School of Computer Science and Engineering, Sun Yat-sen University, Guangzhou 510006, Guangdong, China. Email: zhangh1117@yeah.net.}
}

\date{}

\begin{document}

\maketitle

\begin{abstract}
Numerical experiments reported by Quan and Wu [SIAM J Numer Anal 61 (2023) 2106--2132] show that the observed temporal convergence rates of nonuniform L2 methods for subdiffusion models are not consistent with the theoretically predicted order \(3-\alpha\). This discrepancy suggests that a more refined analysis is needed and motivates the development of a nonuniform BDF3--L2 method for the subdiffusion equation. To account for the initial solution singularity, we employ the multiscale solution decomposition to decompose the original solution and approximate a smoother unknown variable that satisfies the subdiffusion model with a smoother source term. The resulting formulation, however, involves restrictive high-order boundary conditions on the source term and initial data. To overcome this difficulty, we introduce a spectral truncation technique that requires only slightly stronger regularity of the data and a controllable truncation error. We establish high-order regularity estimates of the solution to the truncated problem and develop a nonuniform BDF3--L2 method for its numerical approximation, based on which we derive a rigorous error estimate of temporal convergence order \(2+\alpha\). Numerical experiments are carried out to substantiate the theoretical findings.

\end{abstract}

\section{Introduction}
\subsection{Main motivation}

We consider the following subdiffusion model, which is shown to provide very competitive descriptions of challenging phenomena such as anomalously diffusive transport through heterogeneous porous media \cite{Alphonse2016,Carrillo2025,Langlands,Lin,Sun}
\begin{align}
   & \partial_t^{\alpha} u(x,t) - \Delta u(x,t)  = f(x,t), \quad (x,t) \in \Omega \times (0, T], \label{eq1.1} \\
   & u(x,0) = u_0(x), \quad x \in \Omega, \quad  u(x,t) = 0, \quad (x,t) \in \partial\Omega \times (0, T]. \label{eq1.2}
\end{align}
Here $\Omega$ is an open bounded domain in $\mathbb R^d$ ($1\leq d\leq 3$) with smooth boundary $\partial \Omega$,  $u_0$ and $f$ are prescribed functions, and the Caputo fractional derivative is defined as  $\partial_t^{\alpha} u(x,t): = \beta_{1-\alpha} * \partial_t u$ for some $0 < \alpha <1$ and  $\beta_{\nu}:={t^{\nu-1}}/{\Gamma(\nu)}$,
where $*$ denotes the symbol of convolution, and $\Gamma(\cdot)$ denotes the Euler' gamma function  \cite{Pod,Zeng2015}.  In addition we denote the Riemann-Liouville fractional integral as  $I_{t}^{\alpha}v := \beta_{\alpha} * v$.

It is well-known that the solution to model \eqref{eq1.1}-\eqref{eq1.2} exhibits weak singularities near the initial time, which would significantly reduce the numerical accuracy of its approximation schemes. To overcome this issue, numerous techniques, e.g., graded and nonuniform meshes, have been developed to recover the desired sharp convergence rate. For instance, L1 schemes on graded meshes were considered by \cite{Kopteva2019,KoptevaMeng,Liao2018,Stynes}, and L2-$1_\sigma$ schemes on graded meshes were discussed by \cite{Chen,KoptevaMeng,Liao2019,Liao2021}, respectively. In addition, correction-based convolution quadratures have also been developed for nonsmooth solutions by \cite{Jin2017,Jin2020}.
%Stynes-O'Riordan-Gracia \cite{Stynes} analyzed L1 on graded meshes, Kopteva \cite{Kopteva2019} provided a multidimensional framework, and Chen-Stynes \cite{Chen} proved second-order convergence of L2-$1_\sigma$ on fitted meshes. Kopteva-Meng \cite{KoptevaMeng}, Kopteva \cite{Kopteva}, Liao-Li-Zhang \cite{Liao2018}, and Liao-McLean-Zhang \cite{Liao2019,Liao2021} established sharp error bounds for L1- and L2-type schemes on general nonuniform meshes.
Temporal high-order schemes, such as L2-type schemes, were developed by \cite{Lv} on uniform meshes and by \cite{Kopteva} on nonuniform meshes.  The  $H^1$-stability and error bounds for general L2-type schemes on nonuniform meshes were established in  \cite{Quan}. Nevertheless, the observed convergence order in its numerical experiments is not consistent with the theoretical estimates, which indicates that more delicate treatments are required to improve the analysis results and thus motivates the current work.

Recently, the work \cite{LiuM} proposed a multiscale solution decomposition (MSD) for nonlocal-in-time problems, the main idea of which  is to decompose the original singular solution into two parts: one part incorporates known terms that capture the multiscale singularity, while the remaining unknown variable  still satisfies the original equation with a smoother  forcing term, and thus
becomes smoother compared with the original singular solution.
% For smooth solutions, uniform-mesh methods such as the L1 scheme, achieving $(2-\alpha)$-order convergence, were analyzed by Langlands-Henry \cite{Langlands}, Sun-Wu \cite{Sun}, and Lin-Xu \cite{Lin}.
% Higher-order L2-type methods include the L2-$1_\sigma$ formula of Alikhanov \cite{Alikhanov},
% the L1-2 method of Gao-Sun-Zhang \cite{Gao}, and another L2-type scheme of Lv-Xu \cite{Lv} achieving optimal $(3-\alpha)$-order;
% Alikhanov-Huang \cite{AlikhanovHuang} extended L2 methods to variable-coefficient problems.
Motivated by this, we follow \cite{LiuM} to separate the solution to \eqref{eq1.1}--\eqref{eq1.2} into the following form
\begin{equation}\label{qwl001}
    u = u_0 + v + I_{t}^{\alpha} \sum_{\ell=0}^{m-1} (I_{t}^{\alpha} \Delta)^{\ell} \tilde{f} = u_0 + v + \sum_{\ell=0}^{m-1} I_{t}^{(\ell+1)\alpha}(\Delta^{\ell} \tilde{f}), \quad m \in \mathbb N^+ ,\quad \tilde f : = f+ \Delta u_0,
\end{equation}
in which the derivatives of the last term of the same order exhibit singularities of different degrees, and the new variable $v$ exhibits stronger smoothness satisfying the subdiffusion equation with different right-hand side terms
\begin{align}
    & \partial_{t}^{\alpha}v(x,t) - \Delta v(x,t) = I_{t}^{m\alpha} \Delta^{m} \tilde{f}(x,t), \quad (x,t) \in \Omega \times (0,T], \label{qwl01} \\
    & v(x,t) = 0, \quad (x,t) \in \partial\Omega \times [0,T], \quad v(x,0) = 0, \quad x \in \Omega.\nonumber
\end{align}
For $m=0$, it is straightforward to obtain $u = u_0 + v$ by \eqref{qwl001}.
 % From the perspective of  analysis, one could instead approximate model \eqref{qwl01} such that the numerical methods based on the  solution smoothness assumptions for subdiffusion model could be applicable \cite{Alikhanov,AlikhanovHuang, Langlands,Lin,Sun},  and then recover $u$ by the relation \eqref{qwl001}.
 %  In this case, the error analysis for $v$ could be directly applied to $u$ by \eqref{qwl001}, and  the remaining work is to analyze the solution regularity for $v$ and error estimates of the numerical scheme for \eqref{qwl01}.

\subsection{Challenges and contributions} \label{sec1:2}

  From the perspective of analysis, one could instead approximate model \eqref{qwl01} such that the numerical methods based on the solution smoothness assumptions for subdiffusion model could be applicable \cite{Alikhanov,AlikhanovHuang, Langlands,Lin,Sun},  and then recover $u$ by the relation \eqref{qwl001}.
  In this case, the error estimates for $v$ could be directly applied to $u$ by \eqref{qwl001}, and the remaining work is to analyze the solution regularity for $v$ and error estimates of the numerical scheme for \eqref{qwl01}.

  Nevertheless,  the introduction of the MSD method to \eqref{eq1.1}--\eqref{eq1.2}  would require comparatively strong boundary conditions  in simplifying the treatment of complex boundary conditions  \cite{LiuM}
  \begin{equation}\label{req}
  f+\Delta u_0 \in C([0,T];\dot H^{2m}(\Omega))
  \end{equation}
  such that $\Delta^k(f +\Delta u_0)=0$ for $0 \leq k \leq m-1$ on $\partial \Omega \times [0, T]$, and $\Delta^m(f +\Delta u_0) \in L^2(\Omega)$ for each $t \in [0, T]$.
To accommodate the strong boundary conditions, we introduce a spectral truncation based on the eigenfunction expansion of the Dirichlet Laplacian \(-\Delta\)  while requiring only slightly stronger regularity assumption of the source term and initial data
and a controllable truncation error. Specifically, we assume
   \begin{equation}\label{req:new}
   f+\Delta u_0 \in C([0,T]; H^{{\hat\gamma}+2m}(\Omega))
   \end{equation}
for some $0<{\hat\gamma}<1/2$ (see \S \ref{sec2-0}).
  In addition,  we prove the high-order regularity estimates of the solution to \eqref{qwl01} to facilitate numerical analysis. Furthermore, we employ the BDF3--L2 method to construct the time-discrete scheme for \eqref{qwl01}, and accordingly develop its fully discrete finite element scheme. The stability analysis of the schemes are analyzed, based on which the optimal error estimate of the fully discrete scheme is proved.

  The rest of the work is organized as follows: In Section \ref{sec2}, we prove the high-order regularity estimates for the solution to \eqref{qwl01} and introduce the truncated subdiffusion model. In Section \ref{sec3},  we prove the stability of the time-semidiscrete scheme. In  Section \ref{sec4}, we prove the error estimate of the time-discrete scheme and the fully discrete scheme,  and numerical examples are carried out to substantiate the theoretical findings in the last section.

\subsection{Notations and preliminaries}

Let $L^p(\Omega)$ with $1\leq p\leq \infty$ be the Banach space of $p$th power Lebesgue integrable functions on $\Omega$. For $0\leq m\in \mathbb{N}$, let $W^{m,p}(\Omega)$ be the Sobolev space of $L^p$ functions with $m$th weak derivatives in $L^p(\Omega)$. All spaces are equipped with standard norms \cite{AdaFou}. In particular, we set $ H^m(\Omega) := W^{m,2}(\Omega) $ and $ H_0^m(\Omega) $ be the closure of $C^\infty_0(\Omega)$ in $H^m(\Omega) $. For a Banach space $\mathcal{X}$ and some $ T > 0 $, let $ W^{m,p}(0,T;\mathcal{X}) $ be the space of functions in $ W^{m,p}(0,T) $ with respect to $\|\cdot\|_{\mathcal{X}}$. We also denote $C^m[0,T]$ as the space of $m$th continuously differentiable functions on $[0,T]$. For simplicity, we denote $\|\cdot\|:=\|\cdot\|_{L^2(\Omega)}$ and omit $\Omega$ in notations of spatial norms.

Let $\{\lambda_i\}_{i=1}^{\infty}$ and $\{\phi_{i}\}_{i=1}^{\infty}$ denote the eigenvalues and orthonormal eigenfunctions of $-\Delta:H^2\cap H^1_0\rightarrow L^2$, where $\{\phi_i\}_{i=1}^{\infty}$ form an orthonormal basis in $L^2(\Omega)$
 and the eigenvalues $\{\lambda_i\}_{i=1}^{\infty}$ form a positive increasing sequence going to $+\infty$. We
define
\begin{align*}
\dot{H}^q:=\Big\{v\in L^2:|v|_{\dot{H}^q}^2:=\left((-\Delta)^q v,v\right)=\sum_{i=1}^\infty\lambda_i^q(v,\phi_i)^2<\infty\Big\},
\end{align*}
equipped with the norm $\| g\| _{\dot{H} ^{q } }: = \big ( \| g\| _{L^{2} }^{2}+ | g| _{\dot{H} ^{q} }^{2}\big) ^{1/ 2}$. It is known that $\dot{H} ^{0} = L^{2}$, $\dot{H} ^2=H^2 \cap H_{0}^{1}$ and for $2p-3/2<q<2p+1/2$ for some $1\leq p\in\mathbb N$, $v\in \dot H^q$ implies $v\in H^q$ and $\Delta^i v=0$ for $0\leq i\leq p-1$ \cite[Appendix 2.4]{Jin}. Throughout this paper, we use $Q$ to denote a positive constant, where $Q$ may assume different values at different occurrences.

By \cite{Jin}, the solution of \eqref{qwl01} could be expressed as follows
\begin{equation}\label{qwl03}
    v = \int_{0}^{t} E(t-s) I_{s}^{m\alpha} \Delta^{m} \tilde{f}(x,s) ds, \quad E(t) \tilde q : = \sum_{j=1}^{\infty}  t^{\alpha-1} E_{\alpha, \alpha}\left(-\lambda_j t^\alpha\right)\left(\tilde q, \phi_j\right)  \phi_j,
\end{equation}
% with
% $$E(t)q : = \sum_{j=1}^{\infty}  t^{\alpha-1} E_{\alpha, \alpha}\left(-\lambda_j t^\alpha\right)\left(q, \varphi_j\right)  \varphi_j   $$
% where the operator $E(t)$ is defined as follows
% % $$
% E(t):=\frac{1}{2 \pi \mathrm{i}} \int_{\Gamma_{\theta, \delta}} e^{z t}\left(z^\alpha+A\right)^{-1} \mathrm{d} z =\sum_{j=1}^{\infty} \int_0^t(t-s)^{\alpha-1} E_{\alpha, \alpha}\left(-\lambda_j(t-s)^\alpha\right)\left(f(\cdot, s), \varphi_j\right) \mathrm{d} s \varphi_j(x),
% $$
 where   $E_{p,q}(\cdot)$ denotes the two-parameter Mittag-Leffler function
% $\Gamma_{\theta, \delta}=\{z \in \mathbb{C}:|z|=\delta,|\arg z| \leq \theta\} \cup\left\{z \in \mathbb{C}: z=\rho e^{ \pm \mathrm{i} \theta}, \rho \geq \delta\right\}$
% for $\delta \in (\frac{\pi}{2}, \pi)$ with $\delta >0$,
and the following estimate holds
\begin{equation}\label{yq0}
    \|E(t)\|_{L^2 \rightarrow L^2} \leq Q t^{\alpha-1}, \quad t>0.   %\quad\|E(t) \psi\|_{\check{H}^s} \leq Q t^{-(s-r) / 2}\|\psi\|_{\check{H}^r}, \quad \psi \in \check{H}^r .
\end{equation}

\section{Solution regularity for truncated model}\label{sec2}

\subsection{Regularity estimates}

Motivated by the discussions in  \S \ref{sec1:2}, it suffices to prove   the regularity estimates of the solution to \eqref{qwl01} in the following theorems.

\begin{theorem}\label{thm1}
  Suppose that \eqref{req} holds, $\|\Delta^{m+2} u_0\|$   is bounded as well as $\|\Delta^{m+1}\partial_{t}f(\cdot,t)\| \leq Q t^{-\sigma}$ for $t>0$ and some $0<\sigma < 1$. Then the following estimate holds
\begin{equation*}
\|\partial_{t}v(\cdot,t)\|_{\dot{H}^{2}} \leq Q t^{(m+1)\alpha - 1}, \quad t \in (0,T].
\end{equation*}

In addition, suppose  $\|\Delta^{m+1}\partial_{t}^{2}f(\cdot,t)\| \leq Q t^{-\sigma}$ for $t>0$, then
\begin{equation*}
\|\partial_{t}^{2}v(\cdot,t)\|_{\dot{H}^{2}}  \leq Q t^{(m+1)\alpha - 2}, \quad t \in (0,T].
\end{equation*}
\end{theorem}

\begin{proof}
We differentiate the solution representation \eqref{qwl03} and apply $I_{t}^{m\alpha} \Delta^{m} \tilde{f}(x,t)|_{t=0} = 0$
to obtain
\begin{align}
    \partial_{t}v &= \int_{0}^{t} E(t-s) \partial_{s} I_{s}^{m\alpha} \Delta^{m} \tilde{f}(x,s) ds \nonumber \\
    &= \int_{0}^{t} E(t-s) \left[ \beta_{m\alpha} \Delta^{m} \tilde{f}(x,0) + I_{s}^{m\alpha} \Delta^{m} \partial_{s} f(x,s) \right] ds. \label{der1}
\end{align}
By the assumptions of the theorem, we apply the norm $\|\cdot\|_{\dot{H}^{2}}$ to both sides of \eqref{der1} and apply the estimate \eqref{yq0} to obtain
\begin{align}
    \|\partial_{t}v\|_{\dot{H}^{2}} &\leq Q \int_{0}^{t} \left\| E(t-s) \left[ \beta_{m\alpha} \Delta^{m} \tilde{f}(\cdot,0) + I_{s}^{m\alpha} \Delta^{m} \partial_{s} f(\cdot,s) \right] \right\|_{\dot{H}^{2}} ds \nonumber \\
    % &\leq Q \int_{0}^{t} (t-s)^{\alpha - 1} \left\| \beta_{m\alpha} \Delta^{m} \tilde{f}(\cdot,0) + I_{s}^{m\alpha} \Delta^{m} \partial_{s} f(\cdot,s) \right\|_{\dot{H}^{2}} ds \nonumber \\
    &\leq Q \int_{0}^{t} (t-s)^{\alpha - 1} \left( \beta_{m\alpha} \| \Delta^{m+1} \tilde{f}(\cdot,0) \| + I_{s}^{m\alpha} \| \Delta^{m+1} \partial_{s} f(\cdot,s) \| \right) ds \nonumber \\
    &\leq Q \int_{0}^{t} (t-s)^{\alpha - 1} \left( \beta_{m\alpha} + I_{s}^{m\alpha} s^{-\sigma} \right) ds \leq Q \int_{0}^{t} (t-s)^{\alpha - 1} s^{m\alpha - 1} ds \leq Q t^{(m+1)\alpha- 1}. \label{qwl04}
\end{align}
By the definition of the Mittag-Leffler function, we evaluate the first right-hand side term of \eqref{der1} as
\begin{align*}
    &\int_{0}^{t} E(t-s) \beta_{m\alpha} \Delta^{m} \tilde{f}(x,0) ds \nonumber \\
    &= \sum_{j=1}^{\infty} \int_{0}^{t} (t-s)^{\alpha - 1} E_{\alpha,\alpha}(-\lambda_{j}(t-s)^{\alpha}) \beta_{m\alpha}(s) ds (\Delta^{m} \tilde{f}(x,0), \phi_{j}) \phi_{j} \nonumber \\
    &= \sum_{j=1}^{\infty} \int_{0}^{t} \sum_{k=0}^{\infty}  \frac{(-\lambda_j)^k(t-s)^{\alpha k + \alpha - 1}}{\Gamma(k\alpha+\alpha)} \frac{s^{m\alpha-1}}{\Gamma(m\alpha)}ds (\Delta^{m} \tilde{f}(x,0), \phi_{j}) \phi_{j} \nonumber \\
    &=  \sum_{j=1}^{\infty} t^{(m+1)\alpha - 1} E_{\alpha,(m+1)\alpha}(-\lambda_{j} t^{\alpha}) (\Delta^{m} \tilde{f}(x,0), \phi_{j}) \phi_{j}.
\end{align*}
Differentiate the above equation with respect to $t$ to get
\begin{align}
    &\partial_{t} \int_{0}^{t} E(t-s) \beta_{m\alpha} \Delta^{m} \tilde{f}(x,0) ds   =  \sum_{j=1}^{\infty} t^{(m+1)\alpha - 2} E_{\alpha,(m+1)\alpha - 1}(-\lambda_{j} t^{\alpha}) (\Delta^{m} \tilde{f}(x,0), \phi_{j}) \phi_{j}. \label{qwl07}
\end{align}
By the decay property of $E_{\alpha,(m+1)\alpha - 1}(-z)$ for $z \to +\infty$, we have $|E_{\alpha,(m+1)\alpha - 1}(-\lambda_{j} t^{\alpha})| \leq Q$ for some constant $Q$ independent from $j$ such that
\begin{equation*}
    \left\| \partial_{t} \int_{0}^{t} E(t-s) \beta_{m\alpha} \Delta^{m} \tilde{f}(x,0) ds \right\|_{\dot{H}^{2}}  \leq Q t^{(m+1)\alpha - 2} \| \Delta^{m+1} \tilde{f}(x,0) \|.
\end{equation*}
We then differentiate the second right-hand side term of \eqref{der1} to get
\begin{equation}\label{qwl06}
    \partial_{t} \int_{0}^{t} E(t-s) I_{s}^{m\alpha} \Delta^{m} \partial_{s} f(x,s) ds = \int_{0}^{t} E(t-s) \partial_{s} I_{s}^{m\alpha} \Delta^{m} \partial_{s} f(x,s) ds.
\end{equation}
Take the norm $\|\cdot\|_{\dot H^2}$ on both sides of \eqref{qwl06} and use \eqref{qwl04} to obtain
\begin{align*}
    \bigg\| \partial_{t} \int_{0}^{t} E(t-s) I_{s}^{m\alpha}  \Delta^{m} \partial_{s} f(x,s) ds \bigg\| \leq Q t^{(m+1)\alpha - 1}.
\end{align*}
We combine the above two estimates to get the estimate of $\partial_{t}^{2}v$, which completes the proof.
\end{proof}

\vskip 1mm
We then prove the high-order regularity estimates of the solution to \eqref{qwl01}.

\begin{theorem}\label{thm2}
Suppose the assumptions in Theorem  \ref{thm1} hold and that  $\|\Delta^{m+1} \partial_t^3 f(\cdot,t)\| \leq Q t^{-\sigma}$ for $t>0$, then
\begin{align}
\|\partial_t^3 v(\cdot,t)\|_{\dot H^2} \le Q t^{(m+1)\alpha-3}, \quad t \in (0,T].\label{thm2:e1}
\end{align}

In addition, supposing  $\|\Delta^{m}\partial_{t}^{4}f(\cdot,t)\| \leq Q t^{-\sigma}$, we have

\begin{equation}\label{thm2:e2}
    \|\partial_{t}^{4}v(\cdot,t)\|  \leq Q t^{(m+1)\alpha - 4}, \quad t \in (0,T].
\end{equation}
%If $\|\Delta^{m}\tilde{f}(\cdot,0)\|< \infty$, $\|\Delta^{m}\partial_{t}^p f(\cdot,0)\| < \infty$ with $p=1,2,3$ and  for some $\sigma < 1$, then it holds that

\end{theorem}

\begin{proof}
We recall from \eqref{der1}, \eqref{qwl07} and \eqref{qwl06} that
\begin{equation}\begin{array}{l}
\partial_t^2 v = \ds \sum_{j=1}^{\infty} t^{(m+1)\alpha - 2} E_{\alpha,(m+1)\alpha - 1}(-\lambda_{j} t^{\alpha}) (\Delta^{m} \tilde{f}(x,0), \phi_{j}) \phi_{j} \\[0.15in]
\ds  \qquad + \int_{0}^{t} E(t-s) \partial_{s} I_{s}^{m\alpha} \Delta^{m} \partial_{s} f(x,s) ds =: P_1 + P_2.\label{yq1}
\end{array}
\end{equation}

To estimate $\p_t^3 v$, we are now in the position to evaluate $\p_t P_1$ and $\p_t P_2$, respectively.
We differentiate $P_1$  with respect to time to obtain
\begin{align}
  \partial_t P_1
   =  \sum_{j=1}^{\infty} t^{(m+1)\alpha - 3} E_{\alpha,(m+1)\alpha - 2}(-\lambda_{j} t^{\alpha}) (\Delta^{m} \tilde{f}(x,0), \phi_{j}) \phi_{j}, \label{kx01}
\end{align}
which follows $\left|E_{\alpha,(m+1)\alpha - 2}(-\lambda_{j} t^{\alpha})\right| \leq Q$ to get
\begin{align}
  \|\partial_t P_1 \|_{\dot H^2}  \leq Q  t^{(m+1)\alpha - 3}  \| \Delta^{m+1} \tilde{f}(\cdot,0)\|, \quad t\in (0,T].
  \label{yq2}
\end{align}
To analyze $\partial_t P_2$, we reformulate  it  in \eqref{yq1} as follows
\begin{align}
  \ P_2 & = \int_{0}^{t} E(t-s) \partial_{s} I_{s}^{m\alpha} \Delta^{m} \partial_s f(x,s) ds \nonumber \\
    &= \int_{0}^{t} E(t-s) \left[ \beta_{m\alpha} \Delta^{m} \partial_s f(x,0) + I_{s}^{m\alpha} \Delta^{m} \partial_s^2 f(x,s) \right] ds, \label{der2}
\end{align}
in which we apply \eqref{qwl07} to rewrite the   right-hand side terms in \eqref{der2} as follows
\begin{equation}\label{kx02}
   \begin{split}
       \partial_t P_2 & =  \sum_{j=1}^{\infty} t^{(m+1)\alpha - 2} E_{\alpha,(m+1)\alpha - 1}(-\lambda_{j} t^{\alpha}) (\Delta^{m} \partial_s f(x,0), \phi_{j}) \phi_{j}  \\
& + \partial_t \int_{0}^{t} E(t-s)I_{s}^{m\alpha} \Delta^{m} \partial_s^2 f(x,s) ds = : P_{21} + P_{22}.
   \end{split}
\end{equation}
We utilize $I_{s}^{m\alpha} \Delta^{m} \partial_s^2 f(x,s)|_{s=0} = 0$ to reformulate $P_{22}$ as follows
\begin{align*}
    P_{22} & = \int_{0}^{t} E(t-s) \partial_{s} I_{s}^{m\alpha} \Delta^{m}\partial_s^2 f(x,s) ds  \\
    & = \int_{0}^{t} E(t-s) \left[ \beta_{m\alpha} \Delta^{m} \partial_s^2f(x,0) + I_{s}^{m\alpha} \Delta^{m} \partial_s^3 f(x,s) \right] ds.
\end{align*}
We follow the assumptions of the theorem, \eqref{kx02} as well as \eqref{qwl04} to further bound
\begin{equation*}
    \|\p_t P_2\|_{\dot  H^2} \le Qt^{(m+1)\alpha - 2} + Q t^{(m+1)\alpha - 1} \le Q t^{(m+1)\alpha - 2},
\end{equation*}
which, combined with \eqref{yq1} and \eqref{yq2}, gives \eqref{thm2:e1}.
% Thus, with given assumptions we get $\|P_{21}\|_{\dot H^2}\leq Q t^{(m+1)\alpha - 2}$ and
% \begin{align*}
%    \|P_{22}\|_{\dot H^2} &\leq Q \int_{0}^{t} \left\| E(t-s) \left[ \beta_{m\alpha} \Delta^{m} \partial_s^2f(\cdot,0) + I_{s}^{m\alpha} \Delta^{m} \partial_s^3 f(\cdot,s) \right] \right\|_{\dot{H}^{2}} ds \nonumber \\
%     &\leq Q \int_{0}^{t} (t-s)^{\alpha - 1} \left\| \beta_{m\alpha} \Delta^{m} \partial_s^2f(\cdot,0) + I_{s}^{m\alpha} \Delta^{m} \partial_s^3 f(\cdot,s) \right\|_{\dot{H}^{2}} ds \nonumber \\
%     &\leq Q \int_{0}^{t} (t-s)^{\alpha - 1} \left( \beta_{m\alpha} \| \Delta^{m+1} \partial_s^2f(\cdot,0) \| + I_{s}^{m\alpha} \| \Delta^{m+1} \partial_s^3f(\cdot,s) \| \right) ds \nonumber \\
%     &\leq Q \int_{0}^{t} (t-s)^{\alpha - 1} \left( \beta_{m\alpha} + I_{s}^{m\alpha} s^{-\sigma} \right) ds \leq Q t^{(m+1)\alpha - 1}.
% \end{align*}

By $\partial_t^4 v = \partial_t^2(P_1+P_2)$, we  further differentiate  \eqref{kx01} and \eqref{kx02} to get
\begin{align*}
    \partial_t^4 v &= \sum_{j=1}^{\infty} t^{(m+1)\alpha - 4} E_{\alpha,(m+1)\alpha - 3}(-\lambda_{j} t^{\alpha}) (\Delta^{m} \tilde{f}(x,0), \phi_{j}) \phi_{j} \\
& \quad  + \sum_{j=1}^{\infty} t^{(m+1)\alpha - 3} E_{\alpha,(m+1)\alpha - 2}(-\lambda_{j} t^{\alpha}) (\Delta^{m} \partial_s f(x,0), \phi_{j}) \phi_{j} \\
& \quad + \partial_t \int_{0}^{t} E(t-s) \left[ \beta_{m\alpha} \Delta^{m} \partial_s^2f(x,0) + I_{s}^{m\alpha} \Delta^{m} \partial_s^3 f(x,s) \right] ds,
\end{align*}
in which we combine \eqref{qwl07},  \eqref{qwl06} and  \eqref{der2}
% \begin{align*}
%     &\partial_{t} \int_{0}^{t} E(t-s) \beta_{m\alpha} \Delta^{m} \partial_s^2f(x,0) ds   =  \sum_{j=1}^{\infty} t^{(m+1)\alpha - 2} E_{\alpha,(m+1)\alpha - 1}(-\lambda_{j} t^{\alpha}) (\Delta^{m} \partial_s^2f(x,0), \phi_{j}) \phi_{j},
% \end{align*}
% and
% \begin{align*}
%     \partial_{t} \int_{0}^{t} E(t-s) I_{s}^{m\alpha} \Delta^{m} \partial_{s}^3 f(x,s) ds &= \int_{0}^{t} E(t-s) \partial_{s} I_{s}^{m\alpha} \Delta^{m} \partial_{s}^3 f(x,s) ds \\
% & = \int_{0}^{t} E(t-s) \left[ \beta_{m\alpha} \Delta^{m} \partial_{s}^3f(x,0) + I_{s}^{m\alpha} \Delta^{m} \partial_{s}^4 f(x,s) \right] ds.
% \end{align*}
to prove \eqref{thm2:e2}. We thus  complete the proof.
\end{proof}

\subsection{Truncated model}\label{sec2-0}

We recall the eigen-expansions $\{\lambda_q, \phi_q\}_{q=1}^{\infty}$ of  $-\Delta : H^2 (\Omega) \cap H_0^1(\Omega) \rightarrow L^2(\Omega)$ to define the truncated version of $g$ as follows
\begin{align}\label{gg}
     \Pi_L g := \sum_{q=1}^L (g,\phi_q) \, \phi_q,\quad  L\in\mathbb N^+
\end{align}
with
\begin{align}
\|g(\cdot, t) -   \Pi_L g(\cdot, t)\|
\le \lambda_{L+1}^{-\hat\gamma/2}\|g(\cdot, t)\|_{\dot H^{\hat\gamma}(\Omega)}
\le Q\lambda_{L+1}^{-\hat\gamma/2}\|g\|_{C([0,T];H^{\hat\gamma}(\Omega))}, \quad t\in[0,T], \label{er:trun}
\end{align}
where we use the equivalence
$C([0,T];H^{\hat\gamma}(\Omega))=C([0,T];\dot H^{\hat\gamma}(\Omega))$
for $0<\hat\gamma< 1/2$ \cite{Jin}.
By choosing a sufficiently large $L$, the truncation error \eqref{er:trun} could be sufficiently small such that the truncated function $\Pi_L g$ provides an accurate approximation to $g$.
Since $\Delta^{l} \phi_q= 0$ on $\partial\Omega\times[0,T]$
for any $0 \le  l \in \mathbb N$,
we thus have $\Delta^{l} \Pi_L g =0$ on
$\partial\Omega\times[0,T]$, which belongs to $C([0,T];\dot{H}^{2l}(\Omega))$.

Motivated by the above discussions, we assume the condition  \eqref{req:new}
% \begin{equation*}
% \tilde f \in C([0,T]; H^{{\hat\gamma}+2m}(\Omega))
% \end{equation*}
% for some $0<{\hat\gamma}<1/2$ and
and  replace $\Delta^m \tilde f$ with its truncated function $\Pi_L(\Delta^m \tilde {f})$  in \eqref{qwl01}. The function $\Pi_L(\Delta^m \tilde {f})$  automatically satisfies $ \Pi_L(\Delta^m \tilde f) \in     C([0, T];\dot H^{2l}(\Omega))$ for any $0 \le l \in \mathbb N$ such that  the  relatively strong constraints in \eqref{req} could be relaxed.
%while the additional cost due to truncation of the solutions can be made arbitrarily smal.
We combine \eqref{qwl001}--\eqref{qwl01} to reconsider the following truncated subdiffusion models
\begin{align}
   & \partial_t^\alpha \hat v(x,t) -  \Delta \hat v(x,t) =  I_t^{m \alpha} \Pi_L (\Delta^{m} \tilde{f})(x,t)\label{yq4} \\
     & \hat v(x,t) = 0, \quad (x,t) \in \partial\Omega \times [0,T], \quad \hat v(x,0) = 0, \quad x \in \Omega \nonumber.
\end{align}
By \eqref{qwl01}, we find out that the error $v-\hat v$ still satisfies the subdiffusion model with the right-hand side term replaced by $I_t^{m \alpha} (\mathrm I-\Pi_L)( \Delta^{m} \tilde f )$.
By the stability estimate for subdiffusion equations \cite[Theorem~6.12]{Jin} and \eqref{er:trun}, we thus obtain
\begin{align}
    \|v-\hat{v}\|
&\leq  Q\Big\| \beta_{\alpha} *  I_t^{m\alpha}  (\mathrm I-\Pi_L)(\Delta^{m}\tilde{f})  \Big\|  \leq Q\, I_t^{(m+1)\alpha} \|(\mathrm I-\Pi_L)(\Delta^{m}\tilde{f})\| \label{er:Solv}\\
&\qquad \qquad \le  Q\lambda_{L+1}^{-{\hat\gamma}/2}\|\tilde f\|_{C([0,T];H^{2m+\hat\gamma}(\Omega))}.\nonumber
\end{align}
We truncation the functions $\Delta^{\ell} \tilde{f}$ for each $0 \le \ell \le m-1$ and follow \eqref{qwl001}  to define $\hat u$ as follows
\begin{align}\label{yq5}
\hat u(x,t): = u_0(x) + \hat v(x,t) + \sum_{\ell=0}^{m-1} I_{t}^{(\ell+1)\alpha}\Pi_L(\Delta^{\ell} \tilde{f})(x,t),
\end{align}
which gives
$$
u-\hat{u}= (v-\hat{v}) + \sum_{\ell=0}^{m-1} I_{t}^{(\ell+1)\alpha}(\mathrm I-\Pi_L)(\Delta^{\ell} \tilde{f})(x,t).
$$
We further have
\begin{align}
    \|u-\hat{u}\| &\leq \|v-\hat{v}\| + \sum_{\ell=0}^{m-1} I_{t}^{(\ell+1)\alpha}\|(\mathrm I-\Pi_L)(\Delta^{\ell} \tilde{f})\| \nonumber \\
    & \leq \sum_{\ell=0}^{m} I_{t}^{(\ell+1)\alpha}\|(\mathrm I-\Pi_L)(\Delta^{\ell} \tilde{f})\| \leq Q\lambda_{L+1}^{-{\hat\gamma}/2}\|  \tilde f\|_{C([0,T];H^{2m+\hat\gamma}(\Omega))}, \label{uuf}
\end{align}
which, combined with \eqref{er:Solv}, demonstrates that  the price paid of the truncation  on the solutions is an additional error of order $Q\lambda_{L+1}^{-{\hat\gamma}/2}\|\tilde f\|_{C([0,T]; H^{2m+\hat\gamma}(\Omega))}$ and could  be made arbitrarily small by choosing a sufficiently large $L$.

% {\color{blue} {We also note that the regularity results in Theorems \ref{thm1}--\ref{thm2} could also be applied to the the truncated model \eqref{yq4}. }}
	
% \begin{lemma}\label{lem1}
%   Suppose that \eqref{req:new} holds, $\|\Delta^{m+2} u_0\|$   is bounded as well as $\|\Delta^{m+1}\partial_{t}f(\cdot,t)\| \leq Q t^{-\sigma}$ for $t>0$ and some $0<\sigma < 1$. Then the following estimate holds
% \begin{equation*}
% \|\partial_{t}v(\cdot,t)\|_{\dot{H}^{2}} \leq Q t^{(m+1)\alpha - 1}, \quad t \in (0,T].
% \end{equation*}

% In addition, suppose   $\|\Delta^{m+1}\partial_{t}^{2}f(\cdot,t)\| \leq Q t^{-\sigma}$ for $t>0$, we have
% \begin{equation*}
% \|\partial_{t}^{2}v(\cdot,t)\|_{\dot{H}^{2}}  \leq Q t^{(m+1)\alpha - 2}, \quad t \in (0,T].
% \end{equation*}
% \end{lemma}

% \begin{lemma}
% Regularity on $\hat v$
% \end{lemma}

% \begin{proof}

% \end{proof}

To facilitate numerical approximation, we apply  the operator $^{RL}\partial^{1-\alpha}_t: = \p_t I_t^{\alpha}$  on  both sides of \eqref{yq4}  to obtain its equivalent  model
\begin{align}
   & \partial_t \hat v(x,t) - \partial_t^{1-\alpha} \Delta \hat v(x,t) = \p_t I_t^{(m+1)\alpha}  \Pi_L (\Delta^{m} \tilde{f})(x,t)=:F(x,t),  \label{kk} \\
    & \hat v(x,t) = 0, \quad (x,t) \in \partial\Omega \times [0,T], \quad \hat v(x,0) = 0, \quad x \in \Omega, \nonumber
\end{align}
which would recover the original model \eqref{yq4} by applying the fractional integral operator $I_t^{1-\alpha}$ on both sides of \eqref{kk}.
% For this model, to ensure the applicability of variational techniques in the subsequent numerical analysis, we impose several necessary conditions on the forcing term $\tilde{f}$, following \cite{LiuM}: Assume that $\tilde{f} \in \dot H^{2m}$ for each $t \in [0,T]$, that $\Delta^j \tilde{f} = 0$ on $\partial\Omega \times [0,T]$ for $0 \le j \le m-1$, and that $\Delta^m \tilde{f} \in L^2$ for each $t \in [0,T]$.
% Consequently, in $d$-dimensional space ($1 \le d \le 3$), a possible choice of $f$ and $u_0$ is
% \begin{align*}
%    u_0 = \prod_{l=1}^{d} \sin\left( \frac{x_l}{2} \right), \quad  f = \prod_{l=1}^{d} \sin\left( x_l \right), \quad  x\in \Omega = (0,2\pi)\times(0,2\pi).
% \end{align*}
% However, the above assumptions impose restrictive regularity and compatibility requirements on the data $f$ and the initial value $u_0$.

\begin{remark}

If the source term $f(x,t)$ is structurally complex in time, the right‑hand side integrals in \eqref{yq4} and \eqref{kk} could not be directly evaluated. In this case, a higher‑order interpolation‑based quadrature (e.g.,  the L2 method), can be applied to approximate these terms such that   the resulting truncation error does not  degrade the overall accuracy.
\end{remark}

\section{Stability of the time-semidiscrete scheme}\label{sec3}
In this section, we construct and analyze the time semi-discrete scheme for problem \eqref{kk}.

\subsection{Construction of the time-semidiscrete scheme}

The time semi-discrete scheme is constructed in the following two steps.

 \textbf{Step 1: L2 method for Caputo  derivative.}
%\subsection{L2 method for Caputo fractional derivative}
%Here, we focus on constructing an L2-type approximation for the Caputo fractional derivative $\partial_{t}^{\gamma}w$.
On a general nonuniform time mesh $0=t_0<t_1<\cdots<t_k<\cdots<t_N$ with the step size $\tau_k=t_k-t_{k-1}$, $k\geq 1$, step size ratio $\rho_k=\tau_k/\tau_{k-1}$, $k\geq 2$, and $\rho_{\text{max}}:=\max \limits_{2\leq k\leq N}\rho_k$, we define
\begin{equation*}
D_{k}^{\gamma}w \approx  \delta_t^{\gamma} w (t_k) := \partial_{t}^{\gamma}(\Lambda^{k}w)(t_{k}),\quad \Lambda^{k}=
\left\{\begin{array}{lll}
\Lambda_{1,1} \quad  &\text{on}~(0,t_{1}) \quad &\text{for}~k=1,\\
\Lambda_{2,j} & \text{on}~(t_{j-1},t_{j}) \quad &\text{for}~1\le j< k,\\
\Lambda_{2,k-1}& \text{on}~(t_{k-1},t_{k}) \quad &\text{for}~j=k>1,
\end{array}\right.
\end{equation*}
in which,
\begin{equation*}
\Lambda_{1,j}w(t)=\frac{t-t_{j}}{t_{j-1}-t_{j}}w^{j-1}+\frac{t-t_{j-1}}{t_{j}-t_{j-1}}w^{j},
\end{equation*}
and
\begin{equation*}
\Lambda_{2,j}w(t)=\frac{(t-t_{j})(t-t_{j+1})}{(t_{j-1}-t_{j})(t_{j-1}-t_{j+1})}w^{j-1}+\frac{(t-t_{j-1})(t-t_{j+1})}{(t_{j}-t_{j-1})(t_{j}-t_{j+1})}w^{j}+\frac{(t-t_{j-1})(t-t_{j})}{(t_{j+1}-t_{j-1})(t_{j+1}-t_{j})}w^{j+1}.
\end{equation*}
Here $\Lambda_{1,j}$ and $\Lambda_{2,j}$ represent the standard linear and quadratic Lagrange interpolation operators associated with their respective interpolation nodes $\Lambda_{1,j}:\{t_{j-1},t_{j}\}$ and $\Lambda_{2,j}:\{t_{j-1},t_{j},t_{j+1}\}$.
We follow \cite{Quan} to discretize  the fractional   derivative $D_{k}^{\gamma}w$  as follows

\begin{equation}\label{e2.6}
\begin{aligned}
&D_{1}^{\gamma}w =\frac{1}{\Gamma(2-\gamma)\tau_{1}^{\gamma}}\delta_{1}w,  \\
&D_{k}^{\gamma}w =\frac{1}{\Gamma(1-\gamma)}\Bigg((c_{k}^{(k)}+c_{k-1}^{(k)})\delta_{k}w-a_{k}^{(k)}\delta_{k-1}w-a_{1}^{(k)}\delta_{1}w+\sum_{j=2}^{k-1}d_{j}^{(k)}\delta_{j}w\Bigg),\quad 2 \le k \le N,
\end{aligned}
\end{equation}
where  $\delta_{j}w :=w^{j}-w^{j-1}$, the coefficients $\{a_j^{(k)}\}_{j=1}^{k-1}$, $\{c_j^{(k)}\}_{j=1}^{k-1}$, $a_k^{(k)}$ and $c_k^{(k)}$ are defined as follows
\begin{equation*}
\begin{aligned}
a_{j}^{(k)}=& \frac{\tau_{j+1}}{\left(1-\gamma\right)\tau_{j}\left(\tau_{j}+\tau_{j+1}\right)}\left(t_{k}-t_{j}\right)^{1-\gamma}-\frac{2\tau_{j}+\tau_{j+1}}{\left(1-\gamma\right)\tau_{j}\left(\tau_{j}+\tau_{j+1}\right)}\left(t_{k}-t_{j-1}\right)^{1-\gamma}\\
& +\frac{2}{\left(2-\gamma\right)\left(1-\gamma\right)\tau_{j}\left(\tau_{j}+\tau_{j+1}\right)}\left[\left(t_{k}-t_{j-1}\right)^{2-\gamma}-\left(t_{k}-t_{j}\right)^{2-\gamma}\right],\\
c_{j}^{(k)}=& \frac{1}{\left(1-\gamma\right)\tau_{j+1}\left(\tau_{j}+\tau_{j+1}\right)}\bigg[-\tau_{j}\left(\left(t_{k}-t_{j-1}\right)^{1-\gamma}+\left(t_{k}-t_{j}\right)^{1-\gamma}\right)\\
&\qquad  \qquad  \qquad \qquad \qquad \qquad  +2\left(2-\gamma\right)^{-1}\left(\left(t_{k}-t_{j-1}\right)^{2-\gamma}-\left(t_{k}-t_{j}\right)^{2-\gamma}\right)\bigg],\\
a_{k}^{(k)}=& \frac{\gamma\tau_{k}^{2}}{\left(2-\gamma\right)\left(1-\gamma\right)\tau_{k-1}\left(\tau_{k-1}+\tau_{k}\right)\tau_{k}^{\gamma}},
c_{k}^{(k)}= \frac{1}{\left(1-\gamma\right)\tau_{k}^{\gamma}}+\frac{\gamma\tau_{k}}{\left(2-\gamma\right)\left(1-\gamma\right)\left(\tau_{k-1}+\tau_{k}\right)\tau_{k}^{\gamma}}
\end{aligned}
\end{equation*}
and $d_{j}^{(k)}:=c_{j-1}^{(k)}-a_{j}^{(k)}$. We then refer to the following lemma for future use.

\begin{lemma}  \cite[Corollary 3.3]{Quan} \label{psd}
Let $\rho_L \approx 0.457333$ and $\rho_R \approx 3.561553$.
If $\rho_k \in [\rho_L, \rho_R]$ for all $k \ge 2$, then the following inequality holds:
\begin{equation*}
B_n(w,w)
:= \sum_{k=1}^{n} \langle D_k^{\gamma} w,\, \delta_k w \rangle
\geq Q \sum_{k=1}^{n} \tau_k^{-\gamma} \|\delta_k w\|^2 \geq 0, \quad n \ge 2,
\end{equation*}
where $Q>0$ is a constant depending only on $\gamma$.
\end{lemma}

 \textbf{Step 2: BDF3 method for the time derivative.} We follow \cite{Calvo,LiZ} to introduce a variable-step BDF3 method to approximate the time derivative $\partial_t w$. We introduce the following auxiliary functions
\begin{align*}
  \sigma_0(y,z) & = \frac{2y+1}{y+1} + \frac{yz}{yz+z+1}, \\
  \sigma_1(y,z) & = -\frac{y}{y+1} - \frac{yz}{yz+z+1} - \frac{yz^2}{yz+z+1}\cdot \frac{y+1}{z+1}, \\
  \sigma_2(y,z) & = \frac{yz^2}{yz+z+1}\cdot \frac{y+1}{z+1},
\end{align*}
which satisfy $\sum_{j=0}^{2}\sigma_j(y,z)=1$. Then we denote some helpful notations
\begin{align*}
   \delta_t w^k := \delta_k w /\tau_k, \quad k \geq 1, \quad  \sigma_j^{(k)} := \sigma_j(\rho_k,\rho_{k-1}),  \quad k \geq 3, \quad j=0,1,2.
\end{align*}
For $k \geq 3$, the variable-step BDF3 scheme is given as follows
\begin{align}\label{qwl002}
   \partial_t w (t_k)  \approx \bar{\delta}_t w^k := \sum_{j=0}^{2}  \sigma_j^{(k)} \delta_t w^{k-j}.
\end{align}
For $k=1,2$, we apply the backward Euler method to approximate $\partial_t w$ at $t=t_k$.

\begin{lemma}\cite[Theorem 2.1]{Qi}\label{BDF3}
   If $\rho_k\in [0.5, 1.7319]$ and fixed $\tilde{\beta}=\frac{173}{200}$, then we have
   \begin{align*}
      2\tau_k \, \delta_t w^k \, \bar{\delta}_t w^k = \mathcal{D}_{k}(q_1,q_2) - \mathcal{D}_{k-1}(q_2,q_3) + \mathcal{S}_k(q_1,q_2,q_3),
   \end{align*}
where $q_{j+1}:=\delta_t w^{k-j}$ for $j=0,1,2$, and it holds that
\begin{align*}
  \mathcal{D}_{k}(a,b) & = d_{1,k}^{*} \tau_k a^2 + d_{2,k}^{*}( \tilde{\beta}\sqrt{\tau_k}a - \sqrt{\tau_{k-1}}b )^2,  \\
  \mathcal{S}_k(a,b,c) & = s_{1,k}^{*} \tau_k a^2 + s_{2,k}^{*}( \tilde{\beta}\sqrt{\tau_k}a - \sqrt{\tau_{k-1}}b )^2 + s_{3,k}^{*}( \sqrt{\tau_k}a - \tilde{\beta}\sqrt{\tau_{k-1}}b + \sqrt{\tau_{k-2}}c )^2. \nonumber
\end{align*}
Here $d_{i,k}^{*}>0$ ($i=1,2$), $s_{j,k}^{*}>0$ ($j=1,2,3$), and $\mathcal{S}_k(a,b,c)\geq s_* \tau_k a^2$ holds with $s_* > 2\times 10^{-4}$.
\end{lemma}

%\subsection{Establishment of time-discrete scheme}

    For $k \ge 1$, we  consider the model \eqref{kk} and the relation \eqref{yq5} at $t=t_k$ to obtain
\begin{align}
   & \partial_t \hat v^k - (\partial_t^{1-\alpha} \Delta \hat v)(x,t_k) =  F^k , \label{yq0b}  \\
   & \hat u^k = u_0 + \hat v^k + \sum_{\ell=0}^{m-1} I_{t}^{(\ell+1)\alpha}\Pi_L(\Delta^{\ell} \tilde{f})(x, t_k),  \label{yq00}
\end{align}
where $\hat v^k:=\hat v(x,t_k)$, $\hat u^k:=\hat u(x,t_k)$, and $F^k:=F(x,t_k)$.
We then approximate \eqref{yq0b} in the following two cases.

\textbf{Case I: \bm{$k = 1, 2$}.} We use the backward-Euler method and the nonuniform L2 method \eqref{e2.6} to obtain
\begin{align}\label{yq02}
    \delta_t \hat v^k - D_{k}^{1-\alpha}(\Delta \hat v) = F^k + (\mathcal{R}_1)^k +  (\mathcal{R}_2)^k,
\end{align}
in which,
\begin{align}
(\mathcal{R}_1)^k  := \delta_t \hat v^k - \partial_t \hat v^k, \label{eq:27} \quad
%= \frac{1}{\tau_k}\int_{t_{k-1}}^{t_k}(t_{k-1}-s)\partial_s^2(\Delta \hat v)(x,s)ds, \quad k=1,2, \label{eq:27} \\
(\mathcal{R}_2)^k  := (\partial_t^{1-\alpha} \Delta \hat v)(x,t_k) - D_{k}^{1-\alpha}(\Delta \hat v).
\end{align}

\textbf{Case II: \bm{$k \ge 3$}.} We use the variable-step BDF3  scheme \eqref{qwl002} and \eqref{e2.6}  to get
\begin{align}\label{yq01}
   \bar{\delta}_t \hat v^k - D_{k}^{1-\alpha}(\Delta \hat v)  =  F^k + (\mathcal{R}_1)^k + (\mathcal{R}_2)^k
% where
% \begin{align}
%   (\mathcal{R}_1)^k & = \bar{\delta}_t \hat v^k - \partial_t \hat v^k, \quad k \geq 3, \label{wl01} \\
%   (\mathcal{R}_2)^k & = (\partial_t^{1-\alpha} \Delta \hat v)(x,t_k) - D_{k}^{1-\alpha}(\Delta \hat v),  \quad k \geq 3.  \label{wl02}
\end{align}
with $(\mathcal{R}_1)^k=\bar{\delta}_t \hat v^k - \partial_t \hat v^k$ for $k\geq 3$ and $(\mathcal{R}_2)^k$ defined in \eqref{eq:27}.
For $k \ge 1$, we omit the truncation errors $(\mathcal{R}_1)^k$ and $(\mathcal{R}_2)^k$, and replace $\hat v^k, \hat u^k$ with their numerical approximations $V^k, U^k$ to get the following time-discrete scheme
\begin{align}
   \delta_t V^k  - D_{k}^{1-\alpha}(\Delta V) & = F^k, \quad k=1,2, \quad V^0  = \hat v(x,0) = 0,  \label{xc01} \\
  \bar{\delta}_t V^k  - D_{k}^{1-\alpha}(\Delta V)  & =  F^k, \quad k \geq 3,    \label{xc02} \\
  U^k = V^k + u_0  + & \sum_{\ell=0}^{m-1} I_{t}^{(\ell+1)\alpha} \Pi_L (\Delta^{\ell} \tilde{f})(x, t_k), \quad k \geq 1.    \label{xc002}
\end{align}

\subsection{Stability analysis}
We present the following theorem to prove the stability of the time-discrete scheme \eqref{xc01}--\eqref{xc002}.

\begin{theorem}\label{thm3.3}
If time-step ratio $\rho_k \in [0.5,1.7319]$, then the solutions of \eqref{xc01}--\eqref{xc002} satisfy
\begin{align}
    \|V^n\| & \leq Q \Big( \sum_{k=1}^{n}\tau_k \|F^k\|^2 \Big)^{1/2}, \quad n\geq 1, \label{yq07}\\
    \|U^n\| & \leq \|u_0\| +   Q\Big( \sum_{k=1}^{n}\tau_k \|F^k\|^2 \Big)^{1/2} + \Big\| \sum_{\ell=0}^{m-1} I_{t}^{(\ell+1)\alpha}\Pi_L (\Delta^{\ell} \tilde{f})(\cdot, t_n) \Big\|,  \quad n\geq 1.\nonumber
\end{align}
\end{theorem}

\begin{proof} We take the inner product of \eqref{xc01}--\eqref{xc02} with $2\delta_k V$, and sum the resulting equation from $k=1$ to $n$ to get
\begin{align*}
   \sum_{k=1}^{2} (\delta_t V^k, 2\delta_k V)  +  \sum_{k=3}^{n} (\bar{\delta}_t V^k, 2\delta_k V)  + 2 \sum_{k=1}^{n}(D_{k}^{1-\alpha}(\nabla V), \delta_k (\nabla V) )  & =  2\sum_{k=1}^{n}\tau_k(F^k, \delta_t V^k).
\end{align*}
We then follow Lemma \ref{psd} and Lemma \ref{BDF3} to get
\begin{align*}
    \mathcal{D}_n(\delta_t V^n,\delta_t V^{n-1})  + s_* \sum_{k=3}^{n} \tau_k \|\delta_t V^k\|^2 + 2 \sum_{k=1}^{2} \tau_k \|\delta_t V^k\|^2  & \leq \mathcal{D}_2(\delta_t V^2,\delta_t V^{1}) +  2\sum_{k=1}^{n}\tau_k(F^k, \delta_t V^k),
\end{align*}
which, combined with the fact that
\begin{align*}
  \mathcal{D}_n(\delta_t V^n,\delta_t V^{n-1}) & \geq d_{1,n}^{*} \tau_n \|\delta_t V^n\|^2 \geq 0,   \\
  2\sum_{k=1}^{n}\tau_k(F^k, \delta_t V^k) & \leq \frac{2}{c_*}\sum_{k=1}^{n}\tau_k \|F^k\|^2 + \frac{c_*}{2}\sum_{k=1}^{n} \tau_k \|\delta_t V^k\|^2, \quad  c_*:= \min\{s_*,2\},
\end{align*}
gives
\begin{align}
  \frac{c_*}{2}\sum_{k=1}^{n} \tau_k \|\delta_t V^k\|^2 \leq  \mathcal{D}_2(\delta_t V^2,\delta_t V^{1}) + \frac{2}{c_*}\sum_{k=1}^{n}\tau_k \|F^k\|^2. \label{xc03}
\end{align}
We combine the discrete Cauchy inequality to  obtain
\begin{align*}
    \|V^n - V^0\|^2 = \Big\|\sum_{k=1}^{n} \tau_k^{1/2} (\tau_k^{1/2} \delta_t V^k) \Big\|^2 \leq  \Big(\sum_{k=1}^{n}\tau_k \Big) \sum_{k=1}^{n} \tau_k \|\delta_t V^k\|^2,
\end{align*}
which accordingly gives
\begin{align*}
   \|V^n\| - \|V^0\| \leq \|V^n - V^0\| \leq \sqrt{t_n} \Big(\sum_{k=1}^{n} \tau_k \|\delta_t V^k\|^2\Big)^{1/2}.
\end{align*}
We then combine \eqref{xc03} and $V^0=0$ to obtain
\begin{align}\label{xc04}
    \|V^n\| \leq  \frac{\sqrt{2t_n}}{\sqrt{c_*}} \Big(  \mathcal{D}_2(\delta_t V^2,\delta_t V^{1}) + \frac{2}{c_*}\sum_{k=1}^{n}\tau_k \|F^k\|^2 \Big)^{1/2}.
\end{align}
Next, we discuss the estimates of $\delta_t V^2$ and $\delta_t V^1$. From \eqref{xc01}, we have
\begin{align*}
   \sum_{k=1}^{2} (\delta_t V^k, \delta_k V) + \sum_{k=1}^{2}(D_{k}^{1-\alpha}(\nabla V), \delta_k (\nabla V) )  & =  \sum_{k=1}^{2}\tau_k(F^k, \delta_t V^k).
\end{align*}
We use Lemmas \ref{psd}--\ref{BDF3} to reformulate  the above inequality to obtain
\begin{align*}
  \mathcal{D}_2(\delta_t V^2,\delta_t V^{1}) \leq Q ( \tau_2 \|\delta_t V^2\|^2 + \tau_1 \|\delta_t V^1\|^2 )  =Q \sum_{k=1}^{2} \tau_k \|\delta_t V^k\|^2  & \leq Q  \sum_{k=1}^{2}\tau_k \|F^k\|^2.
\end{align*}
We incorporate this to rewrite \eqref{xc04} to arrive at \eqref{yq07}.
We invoke  \eqref{xc002}--\eqref{yq07}   to complete the proof.
\end{proof}

\section{Error estimates} \label{sec4}
\subsection{Error estimate of the time-discrete scheme}

We derive error estimates for the time-discrete scheme and the fully discrete Galerkin scheme. We first establish the following three auxiliary lemmas for future use.

\begin{lemma} \label{lem01}
 Suppose that \eqref{req:new} holds, $\|\Delta^{m+2} u_0\|$   is bounded, $\|\Delta^{m+1} \partial_t^3 f(\cdot,t)\| \leq Q t^{-\sigma}$  and $\|\Delta^{m} \partial_t^4 f(\cdot,t)\| \leq Q t^{-\sigma}$  for $t>0$ and some $0<\sigma < 1$ as  well as $m\geq \frac{3}{\alpha}$.
If $\rho_k \leq \rho_{\max}$, then it holds that
    \begin{align*}
          \|(\mathcal{R}_1)^k\|  \leq Q \tau_k^{2+\alpha}, \quad k =1,2.
    \end{align*}

  In addition, suppose $\rho_k \in [0.5,1.7319]$, then we have
    \begin{align*}
         \|(\mathcal{R}_1)^k\| \leq Q \tau_k^{2+\alpha}, \quad k\geq 3.
    \end{align*}
\end{lemma}

\begin{proof}
We prove the theorem in the following two cases.

\textbf{Case I: $\bm {k=1,2}$}.
   From Theorems \ref{thm1}--\ref{thm2}, we have the following regularity results
\begin{align*}
     \|\p_t \Delta \hat v (\cdot,t)\| + t\|\p_t^2 \Delta \hat v (\cdot,t)\| + t^2\|\p_t^3 \Delta \hat v (\cdot,t)\| \leq Qt^{(m+1)\alpha - 1}, \quad t\in (0,T],
\end{align*}
which, combined with $(\mathcal{R}_1)^k$ defined in \eqref{eq:27} and the assumptions of the theorem, gives the following estimate for $k=1,2$
\begin{align*}
     \|(\mathcal{R}_1)^k\|  \leq\frac{ Q }{\tau_k}\int_{t_{k-1}}^{t_k}|t_{k-1}-s| \left\|\partial_s^2\Delta \hat v(\cdot,s)\right\| ds \leq\frac{ Q }{\tau_k}\int_{t_{k-1}}^{t_k}|t_{k-1}-s| \,  s^{1+\alpha} \, ds \leq Q \tau_k^{2+\alpha}.
\end{align*}

\textbf{Case II: $ \bm {k\geq 3}$}.
We follow \cite[Equation~(4.7)]{Liao} to obtain the explicit expression of $(\mathcal{R}_1)^k$ defined below \eqref{yq01}
\begin{equation*}
    (\mathcal{R}_1)^k=\bar{\delta}_t \hat v^k - \partial_t \hat v^k
    = \sum_{i=k-2}^{k} \frac{1}{6\tau_i}
      \int_{t_{i-1}}^{t_i} K_{k,k-i}(t)\, \partial_t^4 \hat v(x,t)\, dt
\end{equation*}
with the involved integral kernels defined as follows
\begin{align*}
    K_{k,0}(t) &=
       \bigl(\sigma^{(k)}_{0}- \rho_k \sigma^{(k)}_{1}\bigr)(t_{k-1}-t)^3
       + \rho_k\bigl(\sigma^{(k)}_{1}- \rho_{k-1} \sigma^{(k)}_{2}\bigr)(t_{k-2}-t)^3
       + \rho_k \rho_{k-1} \sigma^{(k)}_{2}(t_{k-3}-t)^3, \\
    K_{k,1}(t) &=
       \bigl(\sigma^{(k)}_{1}- \rho_{k-1} \sigma^{(k)}_{2}\bigr)(t_{k-2}-t)^3
       + \rho_{k-1} \sigma^{(k)}_{2}(t_{k-3}-t)^3, \\
    K_{k,2}(t) &= \sigma^{(k)}_{2}(t_{k-3}-t)^3 .
\end{align*}

For $k=3$, we use Theorem \ref{thm2} and the assumption of the theorem to obtain
\begin{align*}
    \|(\mathcal{R}_1)^3\| & \leq Q \Big[  \tau_3^2 \int_{t_{2}}^{t_{3}} t^{(m+1)\alpha-4}dt +  \tau_2^2 \int_{t_{1}}^{t_{2}} t^{(m+1)\alpha-4}dt +  \tau_1^{-1} \int_{0}^{t_{1}} t^{(m+1)\alpha-1}dt  \Big] \\
    & \leq Q \sum_{j=1}^{3}\tau_j^{2+\alpha} \leq Q \tau_3^{2+\alpha}.
\end{align*}
For $k>3$,  we have
\begin{align*}
    \|(\mathcal{R}_1)^k\| & \leq Q \Big[  \tau_{k}^2 \int_{t_{k-1}}^{t_{k}} t^{(m+1)\alpha-4}dt +  \tau_{k-1}^2 \int_{t_{k-2}}^{t_{k-1}} t^{(m+1)\alpha-4}dt +  \tau_{k-2}^{2} \int_{t_{k-3}}^{t_{k-2}} t^{(m+1)\alpha-4}dt  \Big] \\
& \leq Q \sum_{j=k-2}^{k}\tau_j^{3}t_{j-1}^{(m+1)\alpha-4} \leq  Q \sum_{j=k-2}^{k}\tau_j^{3}t_{j-1}^{\alpha-1}  \leq  Q \tau_k^{2+\alpha}.
\end{align*}
We combine the above two estimates to finish the proof.
\end{proof}

\begin{lemma} \label{lem03}
Suppose that \eqref{req:new} holds, $\|\Delta^{m+2} u_0\|$ is bounded, $\|\Delta^{m+1} \partial_t^3 f(\cdot,t)\| \leq Q t^{-\sigma}$  and $\|\Delta^{m} \partial_t^4 f(\cdot,t)\| \leq Q t^{-\sigma}$  for $t>0$ and some $0<\sigma < 1$,  $\rho_k \in [\rho_L, \rho_R]$ and $m\geq \frac{3}{\alpha}-1$, then the following estimate holds
    \begin{align*}
         \|(\mathcal{R}_2)^k\| & \leq Q\Big[ \tau_k^{2+\alpha} + \sum_{j=2}^{k-1} \tau_j^3 \int_{t_{j-1}}^{t_j} (t_k-s)^{\alpha-2} ds \Big], \quad k\geq 1.
    \end{align*}
\end{lemma}

\begin{proof}
\textbf{Case I:} For $k=1$, we combine \eqref{eq:27}  to obtain
     \begin{align*}
          (\mathcal{R}_2)^1 =  \int_{0}^{t_1}\beta_{\alpha}(t_1-s) \partial_s\big( \Delta \hat v - \Lambda_{1,1}(\Delta \hat v) \big)(\cdot,s) ds.
     \end{align*}
 We combine the integral remainder expansion of the  linear interpolation to obtain
\begin{align*}
  \partial_t \big( \Delta \hat v - \Lambda_{1,1}(\Delta \hat v) \big) = \frac{1}{t_1} \int_{0}^{t} s\, \partial_s^2\Delta \hat v(\cdot,s)\, ds
 - \frac{1}{t_1} \int_{t}^{t_1} (t_1 - s)\, \partial_s^2\Delta \hat v(\cdot,s)\, ds,
\end{align*}
 which, combined with Theorem \ref{thm1} and the assumptions of the theorem, we have
 $\|\partial_t \big( \Delta \hat v - \Lambda_{1,1}(\Delta \hat v) \big)\| \leq Qt_1^2$. Thus we have
 \begin{align*}
     \|(\mathcal{R}_2)^1\| \leq  Qt_1^2 \int_{0}^{t_1}\beta_{\alpha}(t_1-s)  ds \leq Q \tau_1^{2+\alpha}.
 \end{align*}

\textbf{Case II:} Then, for $k\geq 2$ we rewrite \eqref{eq:27} as
     \begin{align*}
          (\mathcal{R}_2)^k & = \sum_{j=1}^{k-1} \int_{t_{j-1}}^{t_j}\beta_{\alpha}(t_k-s) \partial_s\big( \Delta \hat v - \Lambda_{2,j}(\Delta \hat v) \big)(\cdot,s) ds \\
& \quad + \int_{t_{k-1}}^{t_k}\beta_{\alpha}(t_k-s) \partial_s\big( \Delta \hat v - \Lambda_{2,k-1}(\Delta \hat v) \big)(\cdot,s) ds.
     \end{align*}
On the interval $(t_{0},t_1)$, we see that $\|\partial_s\big( \Delta \hat v - \Lambda_{2,1}(\Delta \hat v) \big)(\cdot,s)\|\leq Qs^{(m+1)\alpha-1}$ (see a similar result from \cite[Equation (5.12)]{Stynes}), thus for $m\geq \frac{3}{\alpha}-1$
\begin{align*}
     \Big\| \int_{0}^{t_1}\beta_{\alpha}(t_k-s) \partial_s\big( \Delta\hat v - \Lambda_{2,1}(\Delta\hat v) \big)(\cdot,s) ds \Big\| & \leq Q(t_k-t_1)^{\alpha-1}   \int_{0}^{t_1} s^{(m+1)\alpha - 1} ds \\
     & \quad \leq Q \tau_1^{\alpha-1} t_1^{(m+1)\alpha} \leq Q \tau_1^{2+\alpha}.
\end{align*}
On the interval $(t_{j-1},t_j)$ with $2\leq j \leq k-1$, we get
\begin{align*}
    \|\big( \Delta \hat v - \Lambda_{2,j}(\Delta\hat v) \big)(\cdot,s)\| = \Big\| \frac{(\partial_s^3\Delta\hat v)(\cdot,\xi)}{6}(s-t_{j-1})(s-t_j)(s-t_{j+1}) \Big\|,
\end{align*}
where $\xi\in (t_{j-1},t_{j+1})$. We then follow Theorem \ref{thm2} and the assumptions of the theorem to get
\begin{align*}
    \Big\|  \int_{t_{j-1}}^{t_j}\beta_{\alpha}(t_k-s) & \partial_s\big( \Delta\hat v - \Lambda_{2,j}(\Delta\hat v) \big)(\cdot,s) ds \Big\|  = \Big\| (\alpha-1)  \int_{t_{j-1}}^{t_j}\frac{(t_k-s)^{\alpha-2}}{\Gamma(\alpha)} \big( \Delta\hat v - \Lambda_{2,j}(\Delta\hat v) \big)(\cdot,s) ds \Big\| \\
 & \leq Q \tau_j^3 t_j^{(m+1)\alpha -3} \int_{t_{j-1}}^{t_j} (t_k-s)^{\alpha-2} ds  \leq Q \tau_j^3 \int_{t_{j-1}}^{t_j} (t_k-s)^{\alpha-2} ds,
\end{align*}
which leads to
\begin{align*}
    \Big\| \sum_{j=2}^{k-1} \int_{t_{j-1}}^{t_j} \beta_{\alpha}(t_k-s) \partial_s\big( \Delta\hat v - \Lambda_{2,j}(\Delta\hat v) \big)(\cdot,s) ds \Big\| \leq Q \sum_{j=2}^{k-1} \tau_j^3 \int_{t_{j-1}}^{t_j} (t_k-s)^{\alpha-2} ds.
\end{align*}
On the interval $(t_{k-1},t_k)$, we have
\begin{align*}
      \|\big( \Delta\hat v - \Lambda_{2,k-1}(\Delta\hat v) \big)(\cdot,s)\| \leq Q |(s-t_{k-2})(s-t_{k-1})(s-t_{k})| \leq Q\tau_k^2 (t_k-s),
\end{align*}
which further gives
\begin{align*}
    \Big\|  \int_{t_{k-1}}^{t_k}\beta_{\alpha}(t_k-s) & \partial_s\big( \Delta\hat v - \Lambda_{2,k-1}(\Delta\hat v) \big)(\cdot,s) ds \Big\| \\
    & = \Big\| (\alpha-1)  \int_{t_{k-1}}^{t_k}\frac{(t_k-s)^{\alpha-2}}{\Gamma(\alpha)} \big( \Delta\hat v - \Lambda_{2,k-1}(\Delta\hat v) \big)(\cdot,s) ds \Big\| \\
 & \leq Q \tau_k^2 \int_{t_{k-1}}^{t_k}\frac{(t_k-s)^{\alpha-1}}{\Gamma(\alpha)} ds \leq Q \tau_k^{2+\alpha}.
\end{align*}
Combining the above analysis, we have
    \begin{align*}
          \|(\mathcal{R}_2)^k\| \leq Q\Big[ \tau_2^{2+\alpha} + \sum_{j=2}^{k-1} \tau_j^3 \int_{t_{j-1}}^{t_j} (t_k-s)^{\alpha-2} ds + \tau_k^{2+\alpha} \Big], \quad k\geq 2.
    \end{align*}
This completes the proof.
\end{proof}

\vskip 1mm
We further give the following convergence result based on the above lemmas.
\begin{theorem}\label{conv1}
   Suppose the assumptions in Lemma \ref{lem01} hold, $\rho_k \in [1,1.7319]$ and $m\geq \frac{3}{\alpha}$, then the following estimates holds for the time-discrete scheme \eqref{xc01}--\eqref{xc002}
    \begin{align*}
         \|\hat u^n-U^n\| \leq \|\hat v^n-V^n\| \leq Q   \tau_n^{2+\alpha}, \quad n\geq 1.
    \end{align*}
\end{theorem}

\begin{proof} Define $(e_1)^k := \hat v^k-V^k$ and $(e_2)^k:= \hat u^k-U^k$. By subtracting \eqref{xc01}--\eqref{xc002} from \eqref{yq02}, \eqref{yq01} and \eqref{yq00}, we get the following error equations
\begin{align}
   \delta_t (e_1)^k  - D_{k}^{1-\alpha}(\Delta (e_1)) & = (\mathcal{R}_1)^k + (\mathcal{R}_2)^k, \quad k=1,2, \quad (e_1)^0 = 0,  \label{zxc01} \\
  \bar{\delta}_t (e_1)^k  - D_{k}^{1-\alpha}(\Delta (e_1))  & =  (\mathcal{R}_1)^k + (\mathcal{R}_2)^k, \quad k \geq 3,    \label{zxc02} \\
  (e_2)^k = (e_1)^k, \quad k \geq 1.    \label{zxc03}
\end{align}
We follow the procedures of Theorem~\ref{thm3.3} to arrive at the following estimate for $ \|(e_1)^n\|$ in
\eqref{zxc01}--\eqref{zxc02}
 \begin{align*}
     \|(e_1)^n\| \leq Q \Big( \sum_{k=1}^{n}\tau_k \|(\mathcal{R}_1)^k+(\mathcal{R}_2)^k\|^2 \Big)^{1/2} \leq Q \Big( \sum_{k=1}^{n}\tau_k \|(\mathcal{R}_1)^k\|^2 \Big)^{1/2} + Q \Big( \sum_{k=1}^{n}\tau_k \|(\mathcal{R}_2)^k\|^2 \Big)^{1/2}.
 \end{align*}
For $\rho_k \leq \rho_{\text{max}}$, we use Lemma \ref{lem01} to yield
\begin{align*}
    \Big( \sum_{k=1}^{n}\tau_k \|(\mathcal{R}_1)^k\|^2 \Big)^{1/2} & =  \Big( \sum_{k=1}^{2}\tau_k \|(\mathcal{R}_1)^k\|^2 \Big)^{1/2} + \Big( \sum_{k=3}^{n}\tau_k \|(\mathcal{R}_1)^k\|^2 \Big)^{1/2} \\
& \leq Q \tau_1^{5/2+\alpha} + Q\tau_2^{5/2+\alpha} + Q \; \max_{3\leq k \leq n} \tau_k^{2+\alpha}.
\end{align*}
Then for $\rho_k \in [1,1.7319]$, we utilize Lemma \ref{lem03} to obtain
\begin{align*}
    \|(\mathcal{R}_2)^k\| \leq Q \Big[ \tau_k^{2+\alpha} + \tau_k^3  \int_{t_{1}}^{t_{k-1}} (t_k-s)^{\alpha-2} ds \Big] \leq Q\tau_k^{2+\alpha}, \quad k\geq 2,
\end{align*}
which further gives
\begin{align*}
    \Big( \sum_{k=1}^{n}\tau_k \|(\mathcal{R}_2)^k\|^2 \Big)^{1/2} & = \Big( \tau_1 \|(\mathcal{R}_2)^1\|^2 \Big)^{1/2} + \Big( \sum_{k=2}^{n}\tau_k \|(\mathcal{R}_2)^k\|^2 \Big)^{1/2}  \leq  Q \tau_n^{2+\alpha}.
\end{align*}
We combine the above two estimates and invoke \eqref{zxc03} to complete the proof.
\end{proof}

\subsection{Error estimate of  the fully discrete Galerkin scheme}

We now construct and analyze a fully discrete Galerkin scheme for
\eqref{xc01}--\eqref{xc002}. Let $\Omega$ be partitioned quasi-uniformly with
mesh diameter $h$, and denote by $S_h$ the finite element space of continuous
piecewise linear functions over this partition. The Ritz projection
$I_h: H_0^1(\Omega)\to S_h$ is defined via
\begin{align*}
  (\nabla (\omega - I_h \omega), \nabla \chi ) = 0, \quad  \text{for all } \chi \in S_h
\end{align*}
with the following approximation property
\begin{align}
   \left\| \partial_t^{\hat q} (\omega - I_h\omega)  \right\|_{L^2(\Omega)}\leq Q h^2  \| \partial_t^{\hat q} \omega  \|_{H^2(\Omega)}, \quad \hat q = 0, 1. \label{yq001}
\end{align}
Integrating  \eqref{yq02} and \eqref{yq01} against a test function
$\chi \in H_0^1(\Omega)$ over $\Omega$ yields the corresponding weak
formulation. Hence, for any $\chi \in H_0^1(\Omega)$ and for
$k = 1,2,\ldots,n$, we obtain
\begin{align}
   (\delta_t \hat v^k, \chi)  + (D_{k}^{1-\alpha}\nabla\hat v, \nabla\chi) & = (F^k, \chi) + ((\mathcal{R}_1)^k + (\mathcal{R}_2)^k, \chi), \quad k=1,2, \quad v^0 = 0,  \nonumber  \\
  (\bar{\delta}_t\hat v^k, \chi)  + (D_{k}^{1-\alpha}\nabla\hat v, \nabla\chi)  & =  (F^k, \chi) + ((\mathcal{R}_1)^k + (\mathcal{R}_2)^k, \chi), \quad k \geq 3, \label{FEM:e1}  \\
  \hat u^k = \hat v^k + u_0  & + \sum_{\ell=0}^{m-1} I_{t}^{(\ell+1)\alpha}\Pi_L(\Delta^{\ell}  \tilde{f})(x, t_k), \quad k \geq 1.   \nonumber
\end{align}
We omit  the local truncation errors to obtain the following fully discrete Galerkin
scheme: find $V_h^k$, $U_h^k \in S_h$  such that  $\forall \chi \in S_h$, we have
\begin{align}
   (\delta_t V_h^k, \chi)  + (D_{k}^{1-\alpha}\nabla V_h, \nabla\chi) & = (F^k, \chi), \quad k=1,2, \quad V_h^0 = 0,  \nonumber \\
  (\bar{\delta}_t V_h^k, \chi)  + (D_{k}^{1-\alpha}\nabla V_h, \nabla\chi)  & =  (F^k, \chi), \quad k \geq 3, \label{FE}  \\
  U_h^k = V_h^k +  I_h u_0 + \sum_{\ell=0}^{m-1} I_{t}^{(\ell+1)\alpha} &\Pi_L(\Delta^{\ell}  \tilde{f})(x, t_k), \quad k \geq 1 \nonumber
\end{align}
for $k = 1,2,\ldots,N$.

\begin{corollary}
    If the time-step ratio $\rho_k \in [0.5,1.7319]$, then the solutions of the fully discrete Galerkin scheme \eqref{FE}  satisfy
\begin{align*}
    \|V_h^n\| & \leq Q \Big( \sum_{k=1}^{n}\tau_k \|F^k\|^2 \Big)^{1/2}, \quad n\geq 1, \\
    \|U_h^n\| & \leq \|U_h^0\| +   Q\Big( \sum_{k=1}^{n}\tau_k \|F^k\|^2 \Big)^{1/2} + \Big\| \sum_{\ell=0}^{m-1} I_{t}^{(\ell+1)\alpha} \Pi_L (\Delta^{\ell}  \tilde{f})(\cdot, t_n) \Big\|,  \quad n\geq 1.
\end{align*}
\end{corollary}
\begin{proof}
 By  choosing
$\chi = V_h^{k}$ in \eqref{FE}, the proof could be carried out by following that of  Theorem~\ref{thm3.3} and thus is omitted for simplicity.
\end{proof}

We next prove the error estimate for the fully discrete scheme \eqref{FE}.
\begin{theorem}\label{thm4.1}
Suppose the assumptions in Theorem~\ref{conv1} hold, then the following error estimate holds
\begin{align*}
    \|\hat u^n - U_h^n\| \le \|\hat v^n - V_h^n\| + \|u_0 - U_h^0\|
    \le Q \bigl( \tau_n^{2+\alpha} + h^2\bigr),
    \quad n \ge 1.
\end{align*}
\end{theorem}

\begin{proof}
  For   convenience, we split the error into $\hat u(t_k)- U_h^k =\zeta^k - \eta ^k$ with $\zeta^k = I_h \hat v(t_k) - V_h^k \in S_h$ and $\eta^k = I_h \hat v(t_k) - \hat v(t_k)$ bounded in \eqref{yq001}. We subtract \eqref{FE} from \eqref{FEM:e1} and take $\chi = \zeta^k$ to arrive at the following error equations
\begin{align}
   (\delta_t \zeta^k, \zeta^k)  + (D_{k}^{1-\alpha}(\nabla \zeta), \nabla\zeta^k) & = ((\mathcal{R}_1)^k + (\mathcal{R}_2)^k + \delta_t \eta^k, \zeta^k), \quad k=1,2, \quad \zeta^0 = 0, \nonumber  \\
  (\bar{\delta}_t \zeta^k, \zeta^k)  + (D_{k}^{1-\alpha}(\nabla \zeta), \nabla\zeta^k)  & =  ((\mathcal{R}_1)^k + (\mathcal{R}_2)^k + \bar{\delta}_t \eta^k, \zeta^k), \quad k \geq 3,  \nonumber \\
   u^k - U_h^k = v^k - V_h^k & + (u_0-I_h u_0),  \quad k \geq 1. \label{wd}
\end{align}
We then follow the proof of Theorem \ref{conv1} to  obtain
 \begin{align*}
     \|\zeta^n\| \leq  Q \tau_n^{2+\alpha} +  Q \Big( \sum_{k=1}^{2}\tau_k \|\delta_t \eta^k\|^2 \Big)^{1/2} + Q \Big( \sum_{k=3}^{n}\tau_k \|\bar{\delta}_t \eta^k\|^2 \Big)^{1/2}.
 \end{align*}
We combine  \eqref{yq001}, the assumptions of the theorem  and
Theorem~\ref{thm1} to obtain
 \begin{align*}
   \|\delta_t \eta^k\| & \leq \frac{1}{\tau_k}\int_{t_{k-1}}^{t_k} \|\partial_t \eta \|dt \leq \frac{Qh^2}{\tau_k}\int_{t_{k-1}}^{t_k} \|\partial_t u(\cdot,  t)\|dt \leq  Q h^2, \quad k =1,2, \\
   \|\bar{\delta}_t \eta^k\| & \leq \sum_{j=0}^{2}  |\sigma_j^{(k)}|  \left\|\delta_t \eta^{k-j}\right\| \leq Q \sum_{j=0}^{2} \left\|\delta_t \eta^{k-j}\right\| \leq Qh^2, \quad 3\leq k \leq n.
 \end{align*}
Thus we have
\begin{align*}
  \|\zeta^n\| \leq Q \tau_n^{2+\alpha} +  Q h^2 \Rightarrow \|\hat v^n-V_h^n\| \leq \|\eta^n\| + \|\zeta^n\| \leq Q \big(   \tau_n^{2+\alpha} +   h^2 \big),
\end{align*}
which, together with \eqref{wd},  completes the proof.
\end{proof}

\begin{remark}
By  Theorem~\ref{thm4.1} and the  spectral truncation error \eqref{uuf}, the error estimate of $U_h^n$ to the solution $u^n:=u(x,t_n)$ of the original subdiffusion problem \eqref{eq1.1}--\eqref{eq1.2} could be
evaluated via the triangle inequality as follows
\begin{align}
    \|u^n-U_h^n\|
    &\leq \|\hat{u}^n-U_h^n\| + \|u^n-\hat{u}^n\| \nonumber\\
    &\leq Q(\tau_n^{2+\alpha}+h^2)
    +\hat Q\lambda_{L+1}^{-\hat{\gamma}/2}
    \|\tilde f\|_{C([0,T];H^{2m+\hat{\gamma}}(\Omega))},
    \quad n\geq 1. \label{err:yql}
\end{align}
    Here the positive constant  $Q$ is independent of
$h$, $\tau_n$, and $L$.

Since the eigenvalues \(\{\lambda_i\}_{i=1}^{\infty}\) form a
nondecreasing sequence satisfying \(\lambda_i\to+\infty\) as
\(i\to\infty\), the spectral truncation error tends to zero as
\(L\to\infty\). In particular, \(L\) may be chosen sufficiently large
such that
\[
\hat Q \lambda_{L+1}^{-\widehat{\gamma}/2}
\|\widetilde{f}\|_{C([0,T];H^{2m+\widehat{\gamma}}(\Omega))}
\leq Q(\tau_n^{2+\alpha}+h^2).
\]
Under this choice of \(L\), the estimate \eqref{err:yql} becomes
\[
\|u^n-U_h^n\|
\leq C\bigl(\tau_n^{2+\alpha}+h^2\bigr),
\qquad n\geq 1.
\]
\end{remark}

\section{Numerical experiments}

%\subsection{Setup of graded meshes}

We present several numerical experiments to validate the theoretical analysis. We define the following \emph{graded mesh} of the temporal interval $[0, T]$
\begin{align*}
    t_k = (k\tau)^r, \quad \tau=\frac{T^{1/r}}{N}, \quad k=0, 1, \cdots, N, \quad %\quad  \tau_k = t_k - t_{k-1}, \quad k =1,2,\cdots, N,
     1 \le r \le \log_{2}(2.7319)\approx 1.4499
    % \rho_k
    % = \frac{k^{r}-(k-1)^{r}}
    % {(k-1)^{r}-(k-2)^{r}}, \quad k\ge2.
\end{align*}
such that the time-step ratio
\[
    \rho_k = \frac{\tau_k}{\tau_{k-1}}
    = \frac{k^{r}-(k-1)^{r}}
    {(k-1)^{r}-(k-2)^{r}}, \quad k\ge2,
\]
 which is  a monotonically decreasing function  with respect to $k$, automatically satisfies
$\rho_k \in [1,1.7319]$ in Theorem \ref{conv1}.
%The spatial interval is uniformly partitioned using $h=(b-a)/J$ for some integer $J>0$ for $\Omega=(a,b)$.
For the one-dimensional case, we follow the two-mesh strategy~\cite[Page 107]{Farrell} to define the discrete $L^2$ errors
\begin{equation}\label{error1}
    {\rm Error}^1_N(\tau,h)=
    \Big(h\sum_{j=1}^{J-1} (U_j^{2N}-U_j^{N})^2\Big)^{1/2}, \quad
    {\rm Error}^1_J(\tau,h)=
    \Big(h\sum_{j=1}^{J-1} (U_{2j}^{N}-U_j^{N})^2\Big)^{1/2},
\end{equation}
and accordingly define the convergence rates as follows
\begin{equation}\label{order}
    \text{Rate}^1_{N}
    = \log_{2}\!\left(\frac{{\rm Error}^1_N}{{\rm Error}^1_{2N}}\right), \quad
    \text{Rate}^1_{J}
    = \log_{2}\!\left(\frac{{\rm Error}^1_J}{{\rm Error}^1_{2J}}\right).
\end{equation}
 For the two-dimensional  case, we could follow \eqref{error1} and \eqref{order} to accordingly define
 the discrete $L^2$ errors, i.e.,  ${\rm Error}^2_N(\tau,h)$ and ${\rm Error}^2_J(\tau,h)$ as well as
the convergence rates $\text{Rate}^2_{N}$ and $ \text{Rate}^2_{J}$, respectively.
%  let $h_x=h_y=h=\pi/J$ for some integer $J>0$, and then denote
% \begin{equation*}
%     {\rm Error}^2_N(\tau,h)=
%     \Big(h^2\sum_{i=1}^{J-1}\sum_{j=1}^{J-1} (U_{i,j}^{2N}-U_{i,j}^{N})^2\Big)^{1/2}, \quad
%     {\rm Error}^2_J(\tau,h)=
%     \Big(h^2\sum_{i=1}^{J-1}\sum_{j=1}^{J-1} (U_{2i,2j}^{N}-U_{i,j}^{N})^2\Big)^{1/2},
% \end{equation*}
% and the time-space convergence rates
% \begin{equation*}
%     \text{Rate}^2_{N}
%     = \log_{2}\!\left(\frac{{\rm Error}^2_N}{{\rm Error}^2_{2N}}\right), \quad
%     \text{Rate}^2_{J}
%     = \log_{2}\!\left(\frac{{\rm Error}^2_J}{{\rm Error}^2_{2J}}\right).
% \end{equation*}
Throughout this section, we fix $T=1$,  choose $J = 64$ to test the temporal convergence rates and accordingly choose
$N = 32$ to test the spatial convergence rates, respectively. In addition, we follow Theorem \ref{conv1} to choose $m=\lceil\frac{3}{\alpha}\rceil$, where $\lceil\cdot\rceil$ denotes the ceiling function symbol.

% We choose $m=\lceil\frac{3}{\alpha}\rceil$, where $\lceil\cdot\rceil$ denotes the ceiling function symbol. We set $J=64$ and $N=32$ to isolate the effect of temporal convergence and spatial convergence, respectively.

\subsection{Effects of MSD on numerical accuracy}
\begin{figure}[h]
\centering
\subfigure{
\begin{minipage}[t]{0.45\linewidth}
\centering
\includegraphics[width=\linewidth]{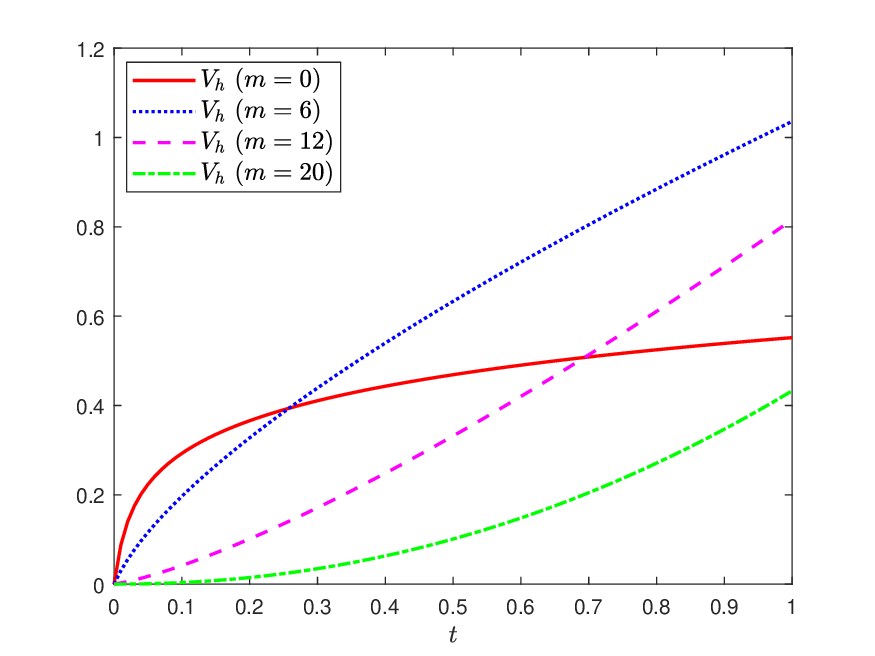}
\end{minipage}%
}
\subfigure{
\begin{minipage}[t]{0.45\linewidth}
\centering
\includegraphics[width=\linewidth]{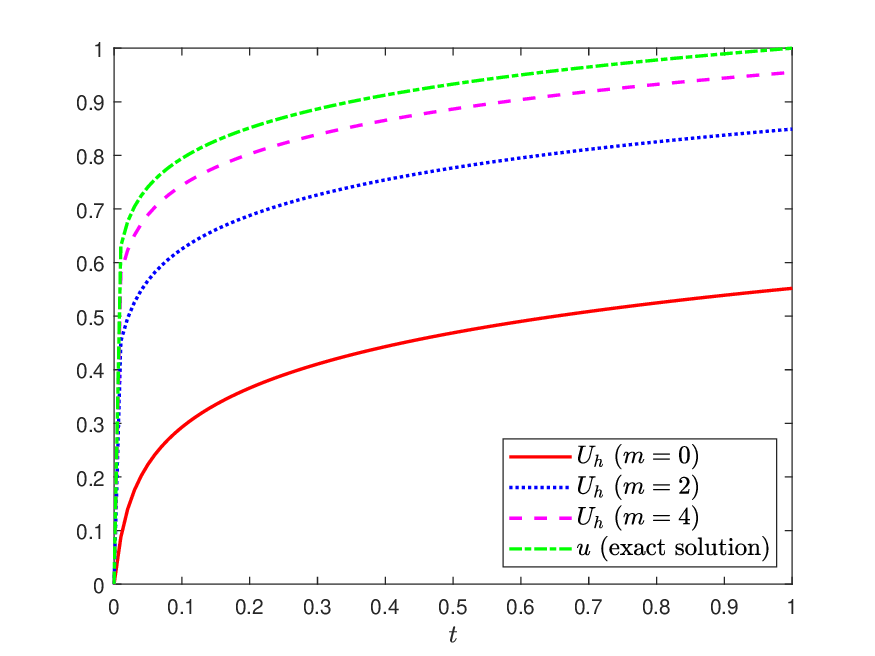}
\end{minipage}%
}%
\centering
\caption{Curves of solutions at $x=\pi/2$.}
\label{Fig1}
\end{figure}

To describe the superiority of MSD, we consider the following problem for illustration:  $\Omega = (0,2\pi)$, $T=1$, $u=t^{0.1}\sin(x)$, and the source term could be accordingly evaluated as $f=(\Gamma(1.1)+t^{0.1})\sin(x)$. We utilize BDF3--L2 method to plot the numerical approximations $V_h$ to \eqref{qwl01} as well as $U_h$ by \eqref{qwl001} at the spatial point $x=\pi /2$ using  the uniform temporal partition $\tau = 10^{-2}$ and the spatial step size $h=\f{\pi}{50}$ in Figure \ref{Fig1}, from which we have the following observations:
\begin{itemize}
\item By \eqref{qwl01}, both $V_h$  and $U_h$ exhibit the initial singularity  for $m=0$, which correspond to the fact that $V_h$  and $U_h$ are weakly singular near the initial time.  As the integer  $m$ increases, the behavior of $V_n$  becomes smoother by the improved regularity of the right-hand side term in \eqref{qwl01}. These observations indicate that the MSD technique could significantly improve the smoothness of the numerical approximation $V_h$ via a larger $m$, as expected in Theorems \ref{thm1}--\ref{thm2}.

\item Based on the improved regularity for $V_h$, the numerical accuracy of  $U_h$  with respect to the solution $u$ is further enhanced. In particular, the smoother feature of $V_h$ directly contributes to a better approximation of $u$ by $U_h$,  which accordingly gives more reliable and convincing numerical results.
\end{itemize}

\subsection{Comparison of convergence behavior}

\begin{table}
\center \footnotesize
\caption{Temporal errors and convergence rates   for \textbf{Example 1}.} \label{tab1}
%\vspace{-0.05in}
%\resizebox{\textwidth}{!}{
\begin{tabular}{cccccccccccc}
\toprule
&  \multicolumn{4}{c}{Our scheme}        &  \multicolumn{2}{c}{Scheme \cite{Quan}}  \\
\cmidrule{3-4}  \cmidrule{6-7}
& $N$  & ${\rm Error}^1_N(\tau,h)$ & $\text{Rate}^1_{N}$ & $N$  & ${\rm Error}^1_N(\tau,h)$ & $\text{Rate}^1_{N}$  \\
\midrule
                    &   512  &  $1.2067 \times 10^{-7}$ &  *      &  512  & $1.3336 \times 10^{-4}$  &  *     \\
  $\alpha =0.35$    &   1024 &  $2.2883 \times 10^{-8}$ &  2.40   &  1024 & $6.6358 \times 10^{-5}$  &  1.01  \\
                    &   2048 &  $4.3987 \times 10^{-9}$ &  2.38   &  2048 & $3.3050 \times 10^{-5}$  &  1.01  \\
                    &   4096 &  $8.5242 \times 10^{-10}$&  2.37   &  4096 & $1.6473 \times 10^{-5}$  &  1.00  \\
\midrule
                    &   512  &  $6.8798 \times 10^{-9}$ &  *      &  512  & $6.0489 \times 10^{-5}$  &  *     \\
  $\alpha =0.95$    &   1024 &  $8.5999 \times 10^{-10}$ &  3.00  &  1024 & $2.9634 \times 10^{-5}$  &  1.03  \\
                    &   2048 &  $1.0735 \times 10^{-10}$ &  3.00  &  2048 & $1.4661 \times 10^{-5}$  &  1.02  \\
                    &   4096 &  $1.3390 \times 10^{-11}$ &  3.00  &  4096 & $7.2909 \times 10^{-6}$  &  1.01  \\
\bottomrule
\end{tabular}
%    }
\end{table}

We conduct numerical examples to verify the theoretical findings and to compare the convergence behavior of the scheme \eqref{FE} with that in   \cite{Quan}.

\textbf{Example 1}
Let $\Omega=(0,2\pi)$. We choose the source term $f =\sin(x)$ and the initial condition
   $u_0(x)=\sin(x/2)$ which   satisfy the constraints under equation \eqref{qwl01}. In this case, we do not need to apply the truncation technique in subsection \ref{sec2-0} and  could directly apply \eqref{FE} to  approximate \eqref{kk} with $\Pi_L(\Delta^m \tilde f)$ on its right-hand side replaced by $\Delta^m \tilde f$. We compare the convergence behavior of the  BDF3--L2  scheme  \eqref{FE}  with the L2 method    \cite{Quan} on the uniform  mesh ($r=1$) and then present the   numerical results   in Table \ref{tab1}.

\begin{table}[h]
\center \footnotesize
\caption{Temporal errors and convergence rates   for \textbf{Example 2}.} \label{tab2}
%\vspace{-0.05in}
%\resizebox{\textwidth}{!}{
\begin{tabular}{cccccccccccc}
\toprule
&  \multicolumn{4}{c}{Our scheme}        &  \multicolumn{2}{c}{Scheme in \cite{Quan}}  \\
\cmidrule{3-4}  \cmidrule{6-7}
& $N$  & ${\rm Error}^2_N(\tau,h)$ & $\text{Rate}^2_{N}$ & $N$  & ${\rm Error}^2_N(\tau,h)$ & $\text{Rate}^2_{N}$  \\
\midrule
                    &   512  &  $1.0396 \times 10^{-4}$ &  *      &  512  & $1.9994 \times 10^{-4}$  &  *     \\
  $\alpha =0.35$    &   1024 &  $1.9969 \times 10^{-5}$ &  2.38   &  1024 & $9.9508 \times 10^{-5}$  &  1.01  \\
                    &   2048 &  $3.8683 \times 10^{-6}$ &  2.37   &  2048 & $4.9569 \times 10^{-5}$  &  1.01  \\
                    &   4096 &  $7.5323 \times 10^{-7}$ &  2.36   &  4096 & $2.4710 \times 10^{-5}$  &  1.00  \\
\midrule
                    &   512  &  $1.2553 \times 10^{-7}$ &  *      &  512  & $8.5699 \times 10^{-5}$  &  *     \\
  $\alpha =0.95$    &   1024 &  $1.5642 \times 10^{-8}$ &  3.00   &  1024 & $4.2149 \times 10^{-5}$  &  1.02  \\
                    &   2048 &  $1.9469 \times 10^{-9}$ &  3.01   &  2048 & $2.0894 \times 10^{-5}$  &  1.01  \\
                    &   4096 &  $2.4219 \times 10^{-10}$&  3.01   &  4096 & $1.0400 \times 10^{-5}$  &  1.01  \\
\bottomrule
\end{tabular}
%    }
\end{table}
We observe from Table \ref{tab1} that under the uniform  mesh, the MSD-based BDF3--L2  scheme  \eqref{FE}
achieves the temporal accuracy of order $2+\alpha$, which substantiates the theoretical findings in Theorem \ref{thm4.1}.
In contrast, the L2 scheme under the uniform mesh exhibits accuracy of only first order  \cite{Kopteva,Quan}. While  the high-order temporal accuracy of order $3-\alpha$ of the scheme in \cite{Kopteva,Quan} could be reached under the condition  $r > 3-\alpha$, the  L2 scheme \cite{Kopteva,Quan} simultaneously increases the numerical difficulties and challenges compared with  the BDF3–L2 scheme \eqref{FE}, which could achieve the high-order temporal accuracy even under the uniform grade.

\begin{table}
\center \footnotesize
\caption{Temporal errors and convergence rates  for \textbf{Example 3}.} \label{tab3}
%\vspace{-0.05in}
%\resizebox{\textwidth}{!}{
\begin{tabular}{cccccccccccc}
\toprule
&  \multicolumn{4}{c}{$r=1$}        &  \multicolumn{2}{c}{$r=1.4$}  \\
\cmidrule{3-4}  \cmidrule{6-7}
& $N$  & ${\rm Error}^1_N(\tau,h)$ & $\text{Rate}^1_{N}$ & $N$  & ${\rm Error}^1_N(\tau,h)$ & $\text{Rate}^1_{N}$  \\
\midrule
                    &  512   &  $1.5697 \times 10^{-6}$ &  *      &  512   & $2.1233 \times 10^{-6}$  &  *     \\
   $\alpha =0.25$   &  1024  &  $3.2271 \times 10^{-7}$ &  2.28   &  1024  & $4.4593 \times 10^{-7}$  &  2.25  \\
                    &  2048  &  $6.6998 \times 10^{-8}$ &  2.27   &  2048  & $9.3659 \times 10^{-8}$  &  2.25  \\
                    &  4096  &  $1.3989 \times 10^{-8}$ &  2.26   &  4096  & $1.9675 \times 10^{-8}$  &  2.25  \\
\midrule
                    &  512   &  $2.2987 \times 10^{-7}$ &  *      & 512   &  $3.2914 \times 10^{-7}$ &  *     \\
   $\alpha =0.5$    &  1024  &  $3.7741 \times 10^{-8}$ &  2.61   & 1024  &  $5.6602 \times 10^{-8}$ &  2.54  \\
                    &  2048  &  $6.3505 \times 10^{-9}$ &  2.57   & 2048  &  $9.7994 \times 10^{-9}$ &  2.53  \\
                    &  4096  &  $1.0860 \times 10^{-9}$ &  2.55   & 4096  &  $1.7057 \times 10^{-9}$ & 2.52  \\
\midrule
                    &  512   & $3.2917 \times 10^{-8}$  &  *      &  512  & $6.3104 \times 10^{-8}$   &  *     \\
   $\alpha =0.95$   &  1024  & $4.1130 \times 10^{-9}$  &  3.00   &  1024 & $7.9035 \times 10^{-9}$   &  3.00  \\
                    &  2048  & $5.1321 \times 10^{-10}$ &  3.00   &  2048 & $9.8707 \times 10^{-10}$  &  3.00  \\
                    &  4096  & $6.3983 \times 10^{-11}$ &  3.00   &  4096 & $1.2309 \times 10^{-10}$  &  3.00  \\
\bottomrule
\end{tabular}
%    }
\end{table}

\begin{table}
\center \footnotesize
\caption{Spatial errors and convergence rates for \textbf{Example 3}.} \label{tab4}
%\vspace{-0.05in}
%\resizebox{\textwidth}{!}{
\begin{tabular}{cccccccccccc}
\toprule
&  \multicolumn{4}{c}{$r=1$}        &  \multicolumn{2}{c}{$r=1.4$}  \\
\cmidrule{3-4}  \cmidrule{6-7}
& $J$  & ${\rm Error}^1_J(\tau,h)$ & $\text{Rate}^1_{J}$ & $J$  & ${\rm Error}^1_J(\tau,h)$ & $\text{Rate}^1_{J}$  \\
\midrule
                    &   64   &  $2.3089 \times 10^{-3}$ &  *      &  64   & $2.3092 \times 10^{-3}$  &  *     \\
   $\alpha =0.25$   &   128  &  $5.9208 \times 10^{-4}$ &  1.96   &  128  & $5.9216 \times 10^{-4}$  &  1.96  \\
                    &   256  &  $1.4651 \times 10^{-4}$ &  2.01   &  256  & $1.4653 \times 10^{-4}$  &  2.01  \\
                    &   512  &  $3.4902 \times 10^{-5}$ &  2.07   &  512  & $3.4907 \times 10^{-5}$  &  2.07  \\
\midrule
                    &   64   &  $2.1356 \times 10^{-3}$ &  *      &  64   & $2.1357 \times 10^{-3}$  &  *     \\
   $\alpha =0.5$    &   128  &  $5.5083 \times 10^{-4}$ &  1.95   &  128  & $5.5086 \times 10^{-4}$  &  1.95  \\
                    &   256  &  $1.3673 \times 10^{-4}$ &  2.01   &  256  & $1.3673 \times 10^{-4}$  &  2.01  \\
                    &   512  &  $3.2637 \times 10^{-5}$ &  2.07   &  512  & $3.2639 \times 10^{-5}$  &  2.07  \\
\midrule
                    &   64   &  $5.8484 \times 10^{-4}$ &  *      &  64   & $5.8486 \times 10^{-4}$  &  *     \\
   $\alpha =0.95$   &   128  &  $1.5368 \times 10^{-4}$ &  1.93   &  128  & $1.5369 \times 10^{-4}$  &  1.93  \\
                    &   256  &  $3.8516 \times 10^{-5}$ &  2.00   &  256  & $3.8517 \times 10^{-5}$  &  2.00  \\
                    &   512  &  $9.2449 \times 10^{-6}$ &  2.06   &  512  & $9.2452 \times 10^{-6}$  &  2.06  \\
\bottomrule
\end{tabular}
%    }
\end{table}

\textbf{Example 2 }  Let $\Omega=(0,2\pi)\times(0,2\pi)$, $f(x_1,x_2)=\sin(x_1)\sin(x_2)$ and $u_0(x_1,x_2)=\sin(\f{x_1}{2})\sin(\f{x_2}{2})$. Analogous to  \textbf{Example 1}, one could    apply \eqref{FE}   to  approximate \eqref{kk} without truncation. We present the numerical results in Table \ref{tab2}, from which we could arrive at similar observations as those in \textbf{Example 1}.

\subsection{Convergence behavior of the scheme (\ref{FE})}

We further carry out some numerical examples to verify the convergence behavior of the fully discrete scheme \eqref{FE}.

\textbf{Example 3} Let the domain $\Omega=(0,\pi)$ with $f= e^{x}$ and initial data $u_0(x)=\sin(x)$, which accordingly gives  $\tilde{f}=e^{x}-\sin(x)$.  We thus arrive at
\begin{align*}
    \Delta^{\ell} \tilde f = e^{x}-(-1)^\ell\sin(x) \neq 0 \text{ on } \partial\Omega\times [0,T], \quad \ell = 0,1,\cdots, m.
\end{align*}

To perform the truncation approximation, we first introduce   the eigenpairs of the operator $-\Delta$ with  homogeneous Dirichlet boundary conditions:
$\lambda_q = q^2$ and $\phi_q(x) = \sqrt{\frac{2}{\pi}} \sin ( q x )$ for  $q=1,2,\cdots$ \cite{Evans}.
We choose a sufficiently large $L=1000$ in \eqref{gg} such that  the truncation error \eqref{er:trun} could be omitted, we thus have
\begin{align*}
 \Pi_L(\Delta^{\ell} \tilde f) = \sum_{q=1}^L (\Delta^{\ell} \tilde f,\phi_q)\phi_q,\quad \ell = 0,1,\cdots, m, \quad  (\Delta^{\ell} \tilde f,\phi_q) = \sqrt{\frac{2}{\pi}}\left(\frac{q\left[1-(-1)^qe^\pi\right]}{1+q^2} - (-1)^\ell\frac{\pi}{2}\delta_{q1}\right),
\end{align*}
where $\delta_{q1}$ is the Kronecker delta.
In Tables \ref{tab3}--\ref{tab4}, we present  the numerical results of  the BDF3--L2 method \eqref{FE} under  $r = 1$ and $r = 1.4$, respectively. We observe from Tables \ref{tab3}--\ref{tab4} that the scheme \eqref{FE} achieves the temporal convergence of order $2+\alpha$ and the spatial convergence of the second order both on the uniform mesh and the graded mesh, which are consistent with the theoretical findings in Theorem \ref{thm4.1}.

\begin{table}
\center \footnotesize
\caption{Temporal errors and convergence rates for \textbf{Example 4}.} \label{tab5}
%\vspace{-0.05in}
%\resizebox{\textwidth}{!}{
\begin{tabular}{cccccccccccc}
\toprule
&  \multicolumn{4}{c}{$r=1$}        &  \multicolumn{2}{c}{$r=1.4$}  \\
\cmidrule{3-4}  \cmidrule{6-7}
& $N$  & ${\rm Error}^2_N(\tau,h)$ & $\text{Rate}^2_{N}$ & $N$  & ${\rm Error}^2_N(\tau,h)$ & $\text{Rate}^2_{N}$  \\
\midrule
                    &  512   &  $1.9254 \times 10^{-3}$ &  *      &  512   & $2.6275 \times 10^{-3}$  &  *     \\
   $\alpha =0.25$   &  1024  &  $3.9769 \times 10^{-4}$ &  2.28   &  1024  & $5.5220 \times 10^{-4}$  &  2.25  \\
                    &  2048  &  $8.2774 \times 10^{-5}$ &  2.26   &  2048  & $1.1603 \times 10^{-4}$  &  2.25  \\
                    &  4096  &  $1.7303 \times 10^{-5}$ &  2.26   &  4096  & $2.4380 \times 10^{-5}$  &  2.25  \\
\midrule
                    &  512   &  $3.7090 \times 10^{-6}$ &  *      & 512   &  $5.4458 \times 10^{-6}$ &  *     \\
   $\alpha =0.5$    &  1024  &  $6.1836 \times 10^{-7}$ &  2.58   & 1024  &  $9.4290 \times 10^{-7}$ &  2.53  \\
                    &  2048  &  $1.0515 \times 10^{-7}$ &  2.56   & 2048  &  $1.6405 \times 10^{-7}$ &  2.52  \\
                    &  4096  &  $1.8119 \times 10^{-8}$ &  2.54   & 4096  &  $2.8656 \times 10^{-8}$ &  2.52  \\
\midrule
                    &  512   & $5.4362 \times 10^{-8}$  &  *      &  512  & $1.0809 \times 10^{-7}$   &  *     \\
   $\alpha =0.95$   &  1024  & $6.7649 \times 10^{-9}$  &  3.01   &  1024 & $1.3477 \times 10^{-8}$   &  3.00  \\
                    &  2048  & $8.4089 \times 10^{-10}$ &  3.01   &  2048 & $1.6760 \times 10^{-9}$   &  3.01  \\
                    &  4096  & $1.0448 \times 10^{-10}$ &  3.01   &  4096 & $2.0814 \times 10^{-10}$  &  3.01  \\
\bottomrule
\end{tabular}
%    }
\end{table}

\textbf{Example 4}  Let the domain $\Omega=(0,\pi)\times(0,\pi)$ with $f = 1+e^{\frac{x_1+x_2}{\sqrt{2}}}$ and  $u_0= \sin(x_1)\sin(x_2)$, then we have $\tilde{f}=1+e^{\frac{x+y}{\sqrt{2}}} -2 \sin(x_1)\sin(x_2)$, which satisfies
\begin{align*}
    \Delta^{\ell} \tilde f  \neq 0 \text{ on } \partial\Omega\times [0,T], \quad \ell = 0,1,\cdots, m.
\end{align*}
Similarly, we adopt the eigen-expansions of the operator $-\Delta$ on the  rectangular domain $\Omega$ \cite{Evans} to obtain  $\lambda_{q,l}=q^2+l^2$, and $\phi_{q,l}(x_1,x_2) = \frac{2}{\pi} \sin(q x_1) \sin(l x_2)$. We choose $L=1000$ to arrive at
\begin{align*}
  \Pi_L(\Delta^{\ell} \tilde f) = \sum_{q=1}^{L}\sum_{l=1}^{L} (\Delta^{\ell} \tilde f,\phi_{q,l}) \, \phi_{q,l},\quad \ell = 0,1,\cdots, m
\end{align*}
with
\begin{align*}
(\Delta^{\ell}\widetilde f,\phi_{q,l})
={}&
\frac{2\delta_{\ell0}}{\pi}
\frac{1-(-1)^q}{q}
\frac{1-(-1)^l}{l}
\\
&+
\frac{2}{\pi}
\frac{q\left(1-(-1)^q e^{\pi/\sqrt2}\right)}
     {q^2+\frac12}
\frac{l\left(1-(-1)^l e^{\pi/\sqrt2}\right)}
     {l^2+\frac12}
+
\frac{\pi}{2}(-2)^{\ell+1}\delta_{q1}\delta_{l1},
\end{align*}
where $\delta_{ij}$ denotes the Kronecker delta. The numerical results are presented in Tables \ref{tab5}--\ref{tab6}, which again show the second-order accuracy in space as well as the  $(2+\alpha)$-th order accuracy  in time of scheme \eqref{FE}, as proved in Theorem \ref{thm4.1}.

\begin{table}
\center \footnotesize
\caption{Spatial errors and convergence rates for \textbf{Example 4}.} \label{tab6}
%\vspace{-0.05in}
%\resizebox{\textwidth}{!}{
\begin{tabular}{cccccccccccc}
\toprule
&  \multicolumn{4}{c}{$r=1$}        &  \multicolumn{2}{c}{$r=1.4$}  \\
\cmidrule{3-4}  \cmidrule{6-7}
& $J$  & ${\rm Error}^2_J(\tau,h)$ & $\text{Rate}^2_{J}$ & $J$  & ${\rm Error}^2_J(\tau,h)$ & $\text{Rate}^2_{J}$  \\
\midrule
                    &   64   &  $5.7040 \times 10^{-1}$ &  *      &  64   & $5.7035 \times 10^{-1}$  &  *     \\
   $\alpha =0.25$   &   128  &  $1.4269 \times 10^{-1}$ &  2.00   &  128  & $1.4268 \times 10^{-1}$  &  2.00  \\
                    &   256  &  $3.5687 \times 10^{-2}$ &  2.00   &  256  & $3.5684 \times 10^{-2}$  &  2.00  \\
                    &   512  &  $8.9261 \times 10^{-3}$ &  2.00   &  512  & $8.9253 \times 10^{-3}$  &  2.00  \\
\midrule
                    &   64   &  $1.1644 \times 10^{-2}$ &  *      &  64   & $1.1645 \times 10^{-2}$  &  *     \\
   $\alpha =0.5$    &   128  &  $3.0854 \times 10^{-3}$ &  1.92   &  128  & $3.0856 \times 10^{-3}$  &  1.92  \\
                    &   256  &  $7.8230 \times 10^{-4}$ &  1.98   &  256  & $7.8234 \times 10^{-4}$  &  1.98  \\
                    &   512  &  $1.9174 \times 10^{-4}$ &  2.03   &  512  & $1.9175 \times 10^{-4}$  &  2.03  \\
\midrule
                    &   64   &  $3.2658 \times 10^{-3}$ &  *      &  64   & $3.2659 \times 10^{-3}$  &  *     \\
   $\alpha =0.95$   &   128  &  $8.8170 \times 10^{-4}$ &  1.89   &  128  & $8.8173 \times 10^{-4}$  &  1.89  \\
                    &   256  &  $2.2486 \times 10^{-4}$ &  1.97   &  256  & $2.2487 \times 10^{-4}$  &  1.97  \\
                    &   512  &  $5.4874 \times 10^{-5}$ &  2.03   &  512  & $5.4876 \times 10^{-5}$  &  2.03  \\
\bottomrule
\end{tabular}
%    }
\end{table}

% \section{Concluding remarks}\label{sec6}

% In summary, we have established a

% In our future work, we will consider the MSD method for the diffusion-wave model by using the nonuniform BDF3--L2 method.

\section*{Declaration}

\vskip 3 mm
\noindent\textbf{Conflict of interest} The authors declare that they have no known competing financial interests or personal relationships that could have appeared to influence the work reported in this paper.

%\vskip 3mm
%\noindent\textbf{Acknowledgments}

\vskip 3 mm
\noindent\textbf{Funding}
 This work was supported by National Natural Science Foundation of China (No.~126011058), China Postdoctoral Science Foundation (No.~2024M762459), Natural Science Foundation of Hubei Province (No. 2025AFB109), and Postdoctor Project of Hubei Province (No. 2025HBBSHCXB021).

\vskip 3 mm
\noindent\textbf{Data Availability} The datasets are available from the corresponding author upon reasonable request.

%%%%%%%%%%%%%%%%%%%%%%%%%%%%%%%%%%%%%%%%%%%%%%%%%%%%%%%%%%%%%%%%%%%%%%%%%%%%%%%%%%%

\end{document}